\documentclass[leqno,11pt]{article}

\usepackage[margin=2.5cm]{geometry}
\usepackage{amsmath}
\usepackage{amsfonts}
\usepackage{array}
\usepackage{enumerate}
\usepackage{graphicx}
\usepackage{amssymb}
\usepackage{amsthm}
\usepackage[sort&compress,comma,square,numbers]{natbib}
\usepackage{xcolor}
\usepackage{chemarr}
\usepackage{hyperref}

\usepackage[version=3]{mhchem}
\usepackage{chemfig} 
\usepackage{babel}
\usepackage{soul}

\newtheorem{theorem}{Theorem}[section]
\newtheorem{proposition}[theorem]{Proposition}
\newtheorem{lemma}[theorem]{Lemma}
\newtheorem{corollary}[theorem]{Corollary}

\theoremstyle{definition}
\newtheorem{definition}[theorem]{Definition}
\newtheorem{remark}[theorem]{Remark}
\newtheorem{example}[theorem]{Example}

\newcommand{\R}{\mathbb{R}}
\newcommand{\N}{\mathbb{N}}

\DeclareMathOperator{\diag}{diag}
\DeclareMathOperator{\im}{im} 
\DeclareMathOperator{\supp}{supp}

\renewcommand{\k}{\kappa}

\newcommand{\norm}[2][\relax]{\ifx#1\relax \ensuremath{\left\Vert#2\right\Vert} \else \ensuremath{\left\Vert#2\right\Vert_{#1}}\fi}

\definecolor{NiceBlue}{rgb}{0.2,0.2,0.75}

\begin{document}

\title{Critical parameters for singular perturbation reductions of chemical reaction networks}

\author{Elisenda Feliu\\Department of Mathematical Sciences, University of Copenhagen\\Universitetsparken 5, 2100
Copenhagen, Denmark\\ \tt{efeliu@math.ku.dk}\\
\\
Sebastian Walcher\\
Lehrstuhl A f\"ur Mathematik, RWTH Aachen\\
52056 Aachen, Germany\\
\tt{walcher@matha.rwth-aachen.de}\\
\\
Carsten Wiuf\\Department of Mathematical Sciences, University of Copenhagen\\Universitetsparken 5, 2100
Copenhagen, Denmark\\ \tt{wiuf@math.ku.dk}\\
}


\maketitle
\begin{abstract} 
We are concerned with polynomial ordinary differential systems that arise from modelling chemical reaction networks. For such systems, which may be of high dimension and may depend on many parameters, it is frequently of interest to obtain a reduction of dimension in certain parameter ranges. Singular perturbation theory, as initiated by Tikhonov and Fenichel, provides a path toward such reductions. In the present paper we discuss parameter values  that lead to singular perturbation reductions (so-called Tikhonov-Fenichel parameter values, or TFPVs).  An algorithmic approach is known, but it is feasible for small dimensions only. Here we characterize conditions for classes of reaction networks for which TFPVs  arise by turning off reactions  (by setting rate parameters to zero), or by removing certain species (which relates to the classical quasi-steady state approach to model reduction).
In particular, we obtain definitive results for the class of complex balanced reaction networks (of deficiency zero) and  first order reaction networks.
\end{abstract}



\section{Introduction}
The modelling of chemical reaction networks  frequently leads to high-dimensional parameter depen\-dent systems of ordinary differential equations (ODEs).  Even in the presence of a well-established structure theory for large classes of reaction networks, reducing the dimension of such systems is desirable for several reasons: From a quantitative perspective in the laboratory, parameter identification is frequently unfeasible for the full system but {might be} possible for a reduced equation. The Michaelis-Menten system and generalizations can be seen as examples of this; see e.g.\ Segel and Slemrod \cite{SSl}, Keener and Sneyd \cite{KeSn}. From a qualitative vantage point, one strategy to prove special features such as the existence of periodic solutions, or multistationarity, is to prove such features for a reduced system and show that they persist for the full system in some parameter range. For a recent example of this strategy, see \cite{FRW}. {Thus it is of general interest to identify parameter domains where a systematic reduction is possible.}

Typically (although not exclusively) the reduction procedures are based on singular perturbation theory as developed by Tikhonov \cite{tikh} and Fenichel \cite{fenichel}. In the present paper, we will discuss singular perturbation reductions and critical parameters that permit reductions of this kind.  The focus will be on characterizing such critical parameters that correspond naturally to structural features of the chemical reaction network.

A frequently used approach to finding appropriate parameters for singular perturbation scenarios goes back to a classical paper by Heineken et al.\ \cite{hta}. The method relies on an adroit scaling of suitable variables (based on an intuitive understanding of the processes in the reaction network) and ideally leads to a system with slow and fast variables to which Tikhonov's and Fenichel's theorems are applicable. From another perspective, a singular perturbation approach for systems with prescribed slow and fast reactions was discussed by Schauer and Heinrich \cite{HeSch}.
More recently, a complete characterization of the parameter values (called Tikhonov-Fenichel parameter values, briefly TFPVs) which give rise to singular perturbations, and of their critical manifolds, was obtained in A.\ Goeke's dissertation \cite{godiss} and the ensuing papers \cite{gw2,gwz,gwz3} by Goeke et al.  Moreover for polynomial or rational systems, an algorithmic path exists toward determining these parameter values. The theory was applied to a number of reaction networks, including  standard reaction networks from biochemistry \cite{KeSn}, and for these all possible singular perturbation reductions could be determined.  In addition, it turned out that the algorithmically determined TFPVs for these systems readily admit an interpretation in terms of chemical species concentrations and reaction rates: Frequently these TFPVs correspond to a ``switching off'' of certain reactions, or a removal of certain chemical species. This is the vantage point for the present paper. Since there is a natural limit to any algorithmic approach for systems with large numbers of variables or parameters, generalizing such structural insights  is of interest.

From a mathematical as well as from a chemical perspective it seems desirable to understand whether (and how) special properties of reaction networks imply the existence of particular classes of TFPVs.
The purpose of the present paper is to contribute toward this understanding. We will focus on  reaction networks with mass-action kinetics, hence on polynomial differential equations.  Our goal is to employ the structure of chemical reaction networks to obtain heuristics for finding TFPV candidates (respectively, candidates for scaling) in a first step and then, in a second step, proceed to verify the TFPV property for some reasonably large and relevant classes of reaction networks.  We make substantial use of the structure theory going back to  Horn and Jackson \cite{horja}, Feinberg \cite{feinberg-balance} and others.  In terms of chemical reaction networks, we are concerned with slow and fast reactions, on the one hand. On the other hand, we investigate the provenance of quasi-steady state phenomena for chemical species, and their naturally associated ``slow-fast'' systems.  Our main results apply in particular to weakly reversible reaction networks of deficiency zero.

Specifically, in Section \ref{sec:CRN_TFPV} we first consider TFPVs that arise from turning off reactions, and identify graphical means for their identification (Theorems \ref{thm_redCB} and \ref{cbthm}). We also provide an explanation of why TFPVs in many cases belong to proper coordinate subspaces (Proposition~\ref{prop:noTFPV}).  In particular, we obtain a complete characterization for weakly reversible systems of deficiency zero. Continuing, in Section \ref{sec:Scale}, we characterise sets of species (so-called LTC species sets) that {``shut down''} the  reaction network when the corresponding variables are zero (hence, the species are present in zero concentration). Such species sets naturally lead to slow-fast systems (in a weak sense), and we further investigate their relation to linear first integrals and give conditions for when an LTC species set is the support of a linear first integral (Proposition \ref{scacor}). We proceed to discuss conditions for TFPV for systems on stoichiometric compatibility classes (Proposition \ref{scctfp} and its corollaries). Finally, we briefly consider combining the approaches to turn off certain reactions and to remove certain species.

\medskip
The paper is organized as follows. Section \ref{sec:prelim} contains preliminaries on reaction networks, TFPVs and (in a weak sense, formally) slow-fast  dynamical systems. Section \ref{sec:CRN_TFPV}, in the context of reaction networks, discusses  TFPVs defined by rate parameters.  Section \ref{sec:Scale} builds on Section \ref{sec:prelim} and connects results of Section \ref{sec:CRN_TFPV} to the classical scaling approach and slow-fast systems. The results are illustrated by examples. In particular, we make use of many standard textbook reaction networks, such as the Michaelis-Menten reaction network  (albeit its reversible version) that serves as benchmarking example.

\section{Preliminaries}\label{sec:prelim}

We let $\R, \R_{\geq 0}, \R_{>0}$ denote the sets of real, non-negative real and positive real numbers,  respectively. Also,  we let  $\N_0$ denote the set of   non-negative integers.
Given $m\in\mathbb N_0$, a {\em coordinate subspace} of $\mathbb R^m$ is defined by $x_{i_1}=\cdots =x_{i_k}=0$ for some $k\in\{0,\ldots,m\}$ and $i_1<\cdots< i_k$. It is \emph{proper} if $k>0$.  The {\em support} $\supp(x)$ of $x\in\mathbb R^m$ is the set of all indices $i$ with $x_i\not=0$, $i=1,\ldots,n$. For $y=(y_1,\ldots,y_n)\in\N_0^n$ and $x=(x_1,\ldots,x_n)\in\R_{\ge 0}^n$, we define $x^y=\prod_{i=1}^n x_i^{y_i}$. If $M=(m_1\ldots m_k)$, $m_i\in\N^n_0$, $i=1,\ldots,k$, is an $(n\times k)$-matrix, then we define $x^M$ as the vector   $(x^{m_1},\ldots,x^{m_k})\in\R^k_{\ge0}$.

\subsection{Reaction networks}\label{crnsubsec}

We consider spatially homogeneous chemical reaction networks   with constant thermodynamical parameters and kinetics of mass-action type. The mathematical theory of these reaction networks was initiated and developed in seminal work by Horn and Jackson \cite{horja}, and Feinberg \cite{feinberg-balance}. We will refer to Feinberg's recent monograph \cite{feinbergbook} as a basic source. First we introduce the notion of a reaction network and fix some terminology.

\begin{definition}\label{crndef}
A \emph{mass-action reaction network} over a set of species $\mathcal{X}=\{X_1,\dots,X_n\}$ is a finite labelled directed graph $G= (\mathcal{Y},\mathcal{R},\k)$ with node set $\mathcal Y$ and edge set $\mathcal R$ such that 
\[ \mathcal{Y} \subseteq  \left\{ \sum_{i=1}^n \alpha_{i}X_i \mid \alpha_i\in\N_0, \ i=1,\ldots,n\right\}\]
 consists of non-negative integer linear combinations in $\mathcal{X}$, and
 $\k$ labels edges by positive real numbers.
 Isolated nodes, but not self-edges, are allowed.
We refer to the nodes as \emph{complexes}, to the edges as \emph{reactions}, and to the labels as \emph{rate parameters}. Every species is assumed to be  in some complex with a  positive coefficient.
 Throughout we let $d$ be the cardinality of $\mathcal{Y}$ and $m$ the cardinality of $\mathcal{R}$. 

A reaction network $\widetilde G$ 
is a \emph{subnetwork} of another reaction network $G$ with species set $\mathcal{X}$, if $\widetilde G$ is a subdigraph of $G$.
\end{definition}

We enumerate the set of complexes in some way, and thus write $Y_j = \sum_{i=1}^n y_{ij} X_i$ with $y_{ij}\in \N_0$. The $y_{ij}$'s are referred to as \emph{stochiometric coefficients}. 
 A labelled reaction between the complexes $Y_j, Y_\ell$ is  written  as
\[ Y_j \cee{->[\k_{\ell j}]} Y_\ell, \qquad \k_{\ell j}>0. \]  
Here $Y_j$ is called a {\em reactant} complex and $Y_\ell$ a {\em product} complex.
Note the reversal of the subindex of $\k$ in the labels.
A numbering of the elements of $\mathcal R$ by $1,\ldots,m$, provides an ordering of $\mathcal{R}$ and we  identify the collection of $\k_{\ell j}$ with a vector $\k \in \mathbb R_{>0}^m$, ordered in the same way as $\mathcal R$, such that $\k_i=\k_{\ell j}$ if $Y_j \cee{->[\k_{\ell j}]} Y_\ell$ is the $i$-th reaction. We will use this convention  without further reference.

The zero complex $0$ is allowed by definition. Reactions with reactant $0$ are called \emph{inflow reactions}, and  account  for production or influx of species.

As a reaction network is given as a directed graph, terminology and properties  from graph theory apply. Moreover, special terminology has been developed, parallel to terminology in graph theory. We will refer to a reaction network where all connected components of the digraph are strongly connected as \emph{weakly reversible}, and otherwise apply standard terminology. 

The evolution of  the species concentrations in time is modelled by means of a system of ODEs,  assuming \emph{mass-action kinetics}.  Denote by $x(t)=(x_1(t),\dots,x_n(t))$ the vector of concentrations of the species $X_1,\dots,X_n$ at time $t$. 
Define the \emph{complex matrix} by
\begin{equation*}\label{massacmat}
Y=\left(y_{ij}\right)_{1\le i\le n,\, 1\le j\le d} \ \in \mathbb R^{n\times d}, 
\end{equation*}
consisting  of the stoichiometric coefficients of the complexes, and let $y_1,\dots,y_d$ denote its columns. 
We let $B$ be the \emph{reactant matrix} with $i$-th column $y_j$ if $Y_j$ is the reactant of the $i$-th reaction,   
and $N\in \R^{n\times m}$  the matrix, referred to as the \emph{stoichiometric matrix}, with $i$-th column given by $y_\ell-y_j$ if $Y_j \cee{->} Y_\ell$ is the $i$-th reaction.  With this notation, the system of ODEs becomes:
\begin{equation}\label{eq:ODE1}
\dot{x}= N \, \diag(\k)\,  x^B, \qquad x\in \R^n_{\geq 0},
\end{equation}
where reference to $t$ is omitted and $\k\in \R^m_{>0}$.  The sets $\R^n_{>0}$ and $\R^n_{\geq 0}$ are positively invariant for \eqref{eq:ODE1}  \cite{volpert}.
Furthermore, there is a useful decomposition of the right-hand side of \eqref{eq:ODE1} in terms of the Laplacian of the reaction network. The \emph{Laplacian matrix} $A(\k) = (a_{ij})_{1\le, i,j\le d}\in \R^{d\times d}$ is given by
\[ a_{ij} = \k_{ij}, \quad i\neq j, \qquad a_{jj} = -\sum_{j\neq \ell} {\k_{\ell j}},\qquad \textrm{for }\quad i,j=1,\ldots,d, \]
where $\k_{ij}=0$ if there is no reaction $Y_j\rightarrow Y_i$. 
Then,   \eqref{eq:ODE1} agrees with
\begin{equation}\label{eq:ODE2}
\dot{x}(t)= Y A(\k)\,  x^Y, \qquad x\in \R^n_{\geq 0},
\end{equation}
(See also Feinberg \cite[Subsection 16.1]{feinbergbook} for further background.)

\medskip
System \eqref{eq:ODE1} often admits {\em stoichiometric first integrals.} These are non-zero linear forms
\[
\phi(x)=\alpha_1x_1+\cdots+\alpha_nx_n
\]
with coefficients $\alpha_i\in\R$, $i=1,\ldots,n$,  such that 
\begin{equation*}\label{stoiconseq}
(\alpha_1,\ldots,\alpha_n)\cdot N=0 \qquad \big(\textrm{equivalently,\quad}(\alpha_1,\ldots,\alpha_n)\cdot Y A(\k)=0\quad \textrm{for all }\k\in\R^m_{>0}\big). 
\end{equation*}
{Note that  $\alpha_1,\ldots,\alpha_n$ might be chosen as integers, since $N$ has integer entries.}

\begin{definition}\label{def:dim}
The image of the stoichiometric matrix $N$  is the {\em stoichiometric subspace}, and the intersection of every coset of this subspace with the {non-negative} orthant is a {\em stoichiometric compatibility class (SCC)}.

The \emph{dimension}   (respectively, \emph{codimension}) of the mass-action reaction network is by definition the dimension (respectively, codimension) of the stoichiometric subspace.
\end{definition}

In principle, system \eqref{eq:ODE1} might admit further linear first integrals. However, the following result 
says that it does not happen for \emph{realistic} networks.

\begin{lemma}[\cite{feinberg-invariant}] \label{lemma:invariant}
If every connected component of a reaction network has exactly one terminal strongly connected component, then every linear first integral of \eqref{eq:ODE1} is stoichiometric.
\end{lemma}

\begin{example}\label{ex:stoich}
Consider the mass-action reaction network with species $X_1,\, X_2$ and reactions
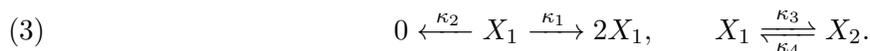
\begin{equation}\label{eq:net1}
0 \cee{<-[\k_2]} X_1 \cee{->[\k_1]} 2X_1, \qquad X_1 \cee{<=>[\k_3][\k_4]}X_2. 
\end{equation}
The corresponding ODE system is given by
\begin{align*}
\dot x_1&=  (\k_1-\k_2) x_1 -\k_3 x_1+\k_4 x_2, &
\dot x_2 &= \k_3x_1-\k_4 x_2.
\end{align*}
The reaction network has no linear first integrals for generic $\k$, but when $\k_1=\k_2$,  the vector $(1,1)$ defines  one. 
This reaction network has one connected component, but two terminal strongly connected components, namely $\{ 0\}$ and $\{2X_1\}$. 

For the network $X_2 \ce{<-[\k_1]}  X_1  \ce{->[\k_2]} X_3$, the dimension is $1$, but {for all $\k$} there are two linearly independent linear first integrals: a stoichiometric linear first integral $\phi_1=x_1+x_2+x_3$, and a non-stoichiometric, $\phi_2=\k_2 x_2 + \k_1 x_3$. 
\end{example}

While it is possible at the outset to reduce the dimension of system \eqref{eq:ODE1} via these first integrals, for the purpose of the present paper it seems appropriate to keep the representation \eqref{eq:ODE1} until at a later stage.

\subsection{Tikhonov-Fenichel parameter values (TFPVs)}\label{tfpvsubsec}

Throughout, when refering to  \emph{singular perturbation reduction}, we mean this in the sense of Tikhonov \cite{tikh}  and Fenichel \cite{fenichel}. In order to identify parameters that give rise to singular perturbation reductions  the following approach was taken in Goeke's dissertation \cite{godiss} and the subsequent papers \cite{gw2,gwz}.

Consider a parameter-dependent ODE {system},
\begin{equation}\label{sys-tpwI}
\dot x = h(x,\pi),\quad x\in \Omega\subseteq  \mathbb R^n,\quad \pi\in\Pi\subseteq\mathbb R^m
\end{equation}
with $h(x,\pi)$ polynomial in $x$ and $\pi$.
We let $D_1h(x,\pi)$ and $D_2 h(x,\pi)$ denote the partial derivatives with respect to $x$ and $\pi$, respectively. Given \(\pi\in \Pi\), we denote by $ \mathcal{V}(h(\cdot,  \pi))$ the zero set of $x\mapsto h(x,\pi)$, and let $n-s^*$ be the generic dimension (with respect to $\pi$) 
 of the vector subspace generated by the entries of $h(x,\pi)$ for $x\in \Omega$. In addition, we require that the generic rank of $D_1h(x,\pi)$, with $x\in\R^n$, equals $n-s^*$ 
 
{In the setting of mass-action reaction networks, this subspace is equal to the stoichiometric subspace under the hypotheses of Lemma~\ref{lemma:invariant}.    
In this case, $s^*$ is the codimension of the reaction network, according to Definition~\ref{def:dim}. 
}

\medskip
The existence of singular perturbation reductions coincides with the existence of  {\em Tikhonov-Fenichel parameter values (TFPVs)}. 

\begin{definition}\label{def:TFPV}
A \emph{TFPV for dimension $s$} ($s^*< s< n$)  of  system  \eqref{sys-tpwI} is a  parameter \(\widehat \pi\in \Pi\), such that the following hold:
\begin{enumerate}[(i)]
 \item The \emph{critical variety} {$\mathcal{V}(h(\cdot, \widehat \pi))\cap \Omega$}   contains {an irreducible} component \(Z\) of dimension \(s\).
 \item There is a Zariski open subset $\widetilde Z\subseteq Z$ such that for all $x\in \widetilde Z$ one has
\[
{\rm rank}\,D_1h(x,\widehat\pi)=n-s
\quad
\textrm{and} 
 \quad
 \mathbb R^n = {\rm Ker}\ D_1h(x,\widehat \pi) \oplus {\rm Im}\ D_1h(x,\widehat \pi).
\]
\item There exists $x_0\in\widetilde Z$ such that all  non-zero eigenvalues of $D_1h(x_0,\widehat \pi)$ have negative real part.
\end{enumerate}

We let $\Pi_s\subseteq\Pi$ denote the set of TFPVs for dimension $s>s^*$. 
\end{definition}

Note that provided (i) holds, then (ii) and (iii) are together equivalent to 
\begin{enumerate}
 \item[(ii')] There exists $x_0\in\widetilde Z$ such that $D_1h(x_0,\widehat \pi)$ has exactly $n-s$ non-zero eigenvalues (counted with multiplicity), which additionally have negative real part.
\end{enumerate}
The conditions imply that the critical manifold $\widetilde Z$ is locally exponentially attracting.
We  have the following characterization  \cite{godiss,gwz}.
 
\begin{proposition}
Given a parameter $\pi\in \Pi$ and any smooth curve $\varepsilon\mapsto \varphi(\varepsilon)$ in the parameter space $\Pi$ with $\varphi(0)=\widehat\pi$, the system
\[\dot x = h(x,\varphi(\epsilon))=h(x,\widehat\pi)+\varepsilon D_2h(x,\widehat\pi)\varphi'(0)+\cdots\]
admits a singular perturbation reduction {in the sense of Tikhonov and Fenichel} if and only if $\pi$ is a TFPV.
\end{proposition}

Thus one may think of a TFPV as a (``degenerate'') parameter set from which singularly perturbed systems emanate.

If $\Pi$ and $\Omega$ are semi-algebraic sets, then  $\Pi_s$ is a semi-algebraic set as well \cite{gwz}. In any case, the Zariski closure of $\Pi_s$ exists and is denoted by 
\begin{equation}\label{eq:Ws}
W_s := \overline{\Pi_s}^{Zar}.
\end{equation}

An alternative characterization of TFPVs and the basis for an algorithmic approach to TFPVs is the following \cite{godiss,gwz}. Consider the characteristic polynomial
\begin{equation*}
\chi(\tau,x,\pi)=\tau^n+\sigma_{1}(x,\pi)\tau^{n-1}+\cdots+ \sigma_{n-1}(x,\pi)\tau+\sigma_n(x,\pi)
\end{equation*}
of  $D_1h(x,\pi)$. Then, given $s^*<s<n$, a parameter value $\widehat\pi$ is a TFPV with locally exponentially attracting critical manifold $\widetilde Z$ (depending on $\widehat \pi$) of dimension $s$, 
if and only  if the following hold for some $x_0 \in \widetilde Z$:
\begin{enumerate}
\item[(iv)] $h(x_0,\widehat\pi)=0$.
\item[(v)] The characteristic polynomial $\chi(\tau,x,\pi)$ satisfies
\begin{enumerate}
\item[(1)] $\sigma_n(x_0,\widehat\pi)=\cdots=\sigma_{n-s+1}(x_0,\widehat\pi)=0$;
\item[(2)] all roots of $\chi(\tau,x_0,\widehat\pi)/\tau^s$ have negative real part.
\end{enumerate}
\item[(vi)] The system $\dot x=h(x,\widehat\pi)$ admits $s$ independent local analytic first integrals at $x_0$.
\end{enumerate}
Therefore,  a starting point for computing TFPVs is as follows: With $h(x_0,\widehat\pi)=0$ and $\sigma_n(x_0,\widehat\pi)=\cdots=\sigma_{n-s+1}(x_0,\widehat\pi)=0$, one sees that $(x_0,\widehat\pi)$ is a solution to $n+s>n$ equations for $x\in\mathbb R^n$, given $\widehat\pi$. In turn, this allows to obtain conditions on $\widehat\pi$ for general polynomial systems via elimination theory.

The validity of the hypotheses for Tikhonov's and Fenichel's theorems depend on the ambient space, and thus may change when passing to an invariant subspace. As a consequence, the notion of TFPV may also  depend on the ambient space. For reaction networks this observation is relevant when passing to {SCCs}.


\subsection{Slow-fast systems and scalings}\label{slofasubsec}

In the present paper,  we call a smooth system of the form
\begin{equation}\label{slofa}
\begin{aligned}
\dot u_1&=  f_1(u_1,u_2,\varepsilon),\\
\dot u_2&= \varepsilon \,f_2(u_1,u_2,\varepsilon),\\
\end{aligned}
\end{equation}
on an open subset of $\mathbb R^s\times\mathbb R^r\times \mathbb R$, with a parameter {$\varepsilon$ in a neighborhood of  $0$}, a {\em slow-fast system}.

A classical approach to a rigorous foundation of quasi-steady state phenomena in chemical reaction networks goes back to Heineken et al. \cite{hta}:  In order to obtain a slow-fast system from \eqref{sys-tpwI} some variables of the system that satisfy  a compatibility condition are scaled  by a small parameter. We outline a simplified version of this technique:
Given {a smooth} curve $\varepsilon \mapsto\pi^*+\varepsilon \rho+\ldots$ in the  parameter space (with $\pi^*$ not necessarily a TFPV), we obtain a system 
\begin{equation}\label{tayloreps}
h(x,\pi^*+\varepsilon \rho+\cdots)=h^{(0)}(x)+\varepsilon h^{(1)}(x)+\varepsilon^2 h^{(2)}(x)+\cdots=:h^*(x,\varepsilon)
\end{equation}
with ``small" parameter $\varepsilon$. Note that
\begin{equation*}\label{eq:h0}
h^{(0)}(x) = h(x,\pi^*).
\end{equation*}

As for the compatibility condition, we follow \cite{lawa}.

\begin{definition}\label{def:indexset}
An index set
\[
\{i_1,\ldots,i_r\},\quad 1\leq i_1<\cdots<i_r\leq n,\quad 1\leq r<n,
\]
is called an {\em LTC index set} for \eqref{tayloreps}, and the set of corresponding variables {$\{x_{i_1},\dots,x_{i_r}\}$} an {\em LTC variable set}, if 
\begin{equation}\label{eq:LCTeqn}
h^{(0)}(x)=0 ,\quad \text{whenever}\quad x_{i_1}=\cdots=x_{i_r}=0.
\end{equation}
 (The acronym stands for ``locally Tikhonov consistent'' \cite{lawa}.) 
 \end{definition}

If the ODE system models a reaction network, then  the corresponding  species set $\{X_{i_1},\ldots,X_{i_r}\}$ is called  a set of {\em LTC species}. If the concentrations of all the species in an LCT set are all zero, then no reaction can take place.
Note that \eqref{eq:LCTeqn} cannot be fulfilled if there are inflow reactions in the reaction network, as $h^{(0)}(x)$ contains a non-zero
 constant monomial.

For an  LCT index set $\{i_1,\dots,i_r\}$, define
\[u_1:=\begin{pmatrix} x_{i_1}\\ \vdots\\ x_{i_r}\end{pmatrix},\]
and collect the remaining variables in $u_2$. Partitioning 
\[
x=\begin{pmatrix}u_1\\ u_2\end{pmatrix},
\]
and rewriting
$h^*(x,\varepsilon)= : \,g(u_1,u_2,\varepsilon)$,
one  obtains a system 
\begin{equation}\label{ltcslofa}
\begin{aligned}
\dot u_1&=  g_1(u_1,u_2,\varepsilon),\\
\dot u_2&= g_2(u_1,u_2,\varepsilon),\\
\end{aligned}
\end{equation}
with $g(0,u_2,\varepsilon)=0$. Scaling $u_1=\varepsilon \, u_1^*$, one can write
$g_i(\varepsilon\,  u_1^*,u_2,\varepsilon)=\varepsilon\,  \widehat g_i(u_1^*,u_2,\varepsilon)$,  
with $\widehat g_1,\,\widehat g_2$ being polynomials, provided that $g_1,\,g_2$ are so, arriving at the slow-fast system as in \eqref{slofa},
\begin{equation}\label{forslofa}
\begin{aligned}
\dot u_1^*&= \widehat g_1(u_1^*,u_2,\varepsilon),\\
\dot u_2&= \varepsilon \,\widehat{g}_2(u_1^*,u_2,\varepsilon).
\end{aligned}
\end{equation}

In the singular perturbation reduction following Heineken et al.\ \cite{hta}, one applies Tikhonov's theorem to \eqref{forslofa}, upon verifying the necessary conditions.  In the literature, a frequently used shortcut is to directly solve $g_1(u_1,u_2,\varepsilon)=0$ (with small $\varepsilon)$ for $u_1$ and substitute the result into the second equation of \eqref{ltcslofa}. We will refer to this procedure as {\em classical QSS reduction.} Note that without  {further analysis, e.g.\ verifying the hypotheses for Tikhonov's theorem}, this is a purely formal procedure.

Obviously, any superset of an LTC species set is also an LTC species set, but minimal LTC species sets are of primary interest: For non-minimal LTC species sets the partial derivative $D_1\widehat g_1$ necessarily has zero columns, so a local resolution of the implicit equation $\widehat g_1=0$ cannot exist.

Since solutions of \eqref{forslofa} are bounded on compact subsets of their maximal existence interval, one finds $u_1=O(\varepsilon)$ on these compact subintervals. But it is not guaranteed that system \eqref{ltcslofa} admits a local $(n-r)$-dimensional invariant manifold close to $u_1=0$ for small positive $\varepsilon$, hence there remains the question whether a singular perturbation reduction exists. Thus, LTC variable sets provide candidates for Tikhonov-Fenichel reductions, but these need further investigation. Moreover, even in the singular perturbation setting, there may not be a connection to TFPVs. We will get back to this later.

In some cases, direct application of Tikhonov-Fenichel does not work, but singular perturbation reduction with a critical variety of higher dimension is possible. For instance, Schneider and Wilhelm \cite{SchWi} considered a scenario, where the fast part of \eqref{forslofa} admits non-trivial first integrals. In such a setting, the partial derivative $D_1\widehat g_1$ cannot have full rank, but if the rank is full on every level set of the first integrals, and the non-zero eigenvalues have negative real parts, then reduction works. (Conversely, the local existence of such first integrals is also necessary  \cite[Prop. 2]{gw2}.)

\medskip 
For  reaction networks it is of interest to understand whether first integrals of system \eqref{ltcslofa} (and thus of \eqref{tayloreps}) carry over, upon scaling, to the fast system at $\varepsilon=0$, and to a possible reduction. We first note an obvious fact.

\begin{lemma} \label{lem:int}
Let $\phi(u_1,u_2,\varepsilon)$ denote a smooth first integral of system \eqref{ltcslofa}. Then, 
\[
\psi(u_1^*,u_2,\varepsilon)=\phi(\varepsilon u_1^*,u_2,\varepsilon)
\]
is a smooth first integral of system \eqref{forslofa}.
\end{lemma}

Letting $\varepsilon\to 0$ in Lemma~\ref{lem:int}, one sees that $\phi(0,u_2,0)$ is (constant or) a first integral of the fast system
\[
\begin{aligned}
\dot u_1^*&= \widehat g_1(u_1^*,u_2,0),\\
\dot u_2&= 0,\\
\end{aligned}
\]
which is a true but uninteresting fact due to $\dot u_2=0$. However, we do have:

\begin{proposition}\label{slofafiprop}
Let $\phi(u_1,u_2,\varepsilon)$ denote a smooth first integral of system \eqref{ltcslofa}. Then, the following hold:
\begin{enumerate}[(a)]
\item If $\phi(0,u_2,0)=0$, then $\phi(\varepsilon u_1^*,u_2,\varepsilon)=\varepsilon \widehat\phi(u_1^*,u_2,\varepsilon)$, and $\widehat \phi$ is a first integral of \eqref{forslofa}, and $\widehat \phi(u_1^*,u_2,0)$ is (constant or) a first integral of the fast system.
\item If system \eqref{forslofa} admits a singular perturbation reduction  
(directly or in the sense of \cite[Prop. 3]{lawa}) then $\widehat \phi(u_1^*,u_2,0)$ is (constant or) a first integral of the reduced system on the critical manifold.
\end{enumerate}
\end{proposition}

\begin{proof} (a) is  straightforward. (b) is shown in \cite[Prop. 4]{lawa}.
\end{proof}

 {
Since stoichiometric first integrals are of particular importance for  reaction networks, we note a special case:
Let $\phi=m_1 u_1+m_2u_2$ be a first integral of \eqref{ltcslofa} with row vectors $m_1,\,m_2$. If $m_2=0$, then $\phi$ is also a first integral of the fast system. (Note that this is a rather restrictive condition.)
 }
\begin{example}\label{mmdisguiseex}
Consider the reversible Michaelis-Menten reaction network \cite{KeSn},
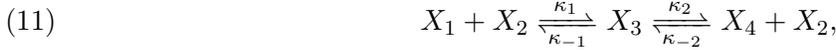
\begin{equation}\label{mmcrn}
X_1+X_2\cee{<=>[\k_1][\k_{-1}]} X_3\cee{<=>[\k_2][\k_{-2}]}  X_4+X_2,
\end{equation}
where $X_1$ chemically is a substrate, $X_2$ an enzyme,  $X_3$ an intermediate complex, and  $X_4$ a product, formed by conversion of the substrate $X_1$.
Using \eqref{eq:ODE1}, we obtain the ODE system
\begin{align*} 
\dot x_1 &= -\k_1x_1x_2 + \k_{-1}x_3, \\
\dot x_2 &= -\k_1x_1x_2 + (\k_{-1}+\k_2)x_3- \k_{-2}x_2x_4,\\
\dot x_3 &=  \k_1x_1x_2 - (\k_{-1}+\k_2)x_3+ \k_{-2}x_2x_4,\\
\dot x_4 &= \k_2x_3 - \k_{-2}x_2x_4.
\end{align*}
The right-hand side vanishes for instance when $x_2=x_3=0$, hence $\{x_2,x_3\}$ is an LTC variable set.   Upon scaling $x_2=\varepsilon\,  x_2^*$ and $x_3=\varepsilon \,  x_3^*$, we obtain the slow-fast system
\begin{align*} 
\dot x_1 &=\varepsilon\,  (-\k_1x_1x_2^* + \k_{-1}x_3^*), \\
\dot x_2^* &= -\k_1x_1x_2^* + (\k_{-1}+\k_2)x_3^*- \k_{-2}x_2^*x_4,\\
\dot x_3^* &= \k_1x_1x_2^* - (\k_{-1}+\k_2)x_3^*+ \k_{-2}x_2^*x_4,\\
\dot x_4 &=\varepsilon\,  ( \k_2x_3^* - \k_{-2}x_2^*x_4).
\end{align*}

Tikhonov's theorem is not directly applicable to this slow-fast system, since the rank condition is not satisfied.  {But a step-by-step approach yields a reduction to dimension one: With the first integral $\phi_1=x_2+x_3$ one obtains a three-dimensional system for $x_1,\,x_2^*$ and $x_4$, for which a singular perturbation reduction to dimension two exists. Then with Proposition \ref{slofafiprop} the first integral $\phi_2=x_1+x_3+x_4$ allows a further reduction to dimension one.}
\end{example}

\section{TFPVs for reaction networks}\label{sec:CRN_TFPV}

\subsection{General considerations}

While the notion of TFPV applies to all parameter dependent polynomial (and more general) vector fields, special properties of reaction networks impose restrictions. 
We give an elementary illustration of this fact.

\begin{example}
Consider the linear differential equation in $\mathbb R^2$,
\[
\dot x =\begin{pmatrix}\alpha_{11}&\alpha_{12}\\ \alpha_{21}&\alpha_{22}\end{pmatrix}x + b =A \,x +b,
\]
where the second equality defines $A$, and  $\alpha_{ij}\in\R$, $i,j=1,2$. By Definition~\ref{def:TFPV}(i), a TFPV satisfies 
\[
0= \det A= \alpha_{11}\alpha_{22}-\alpha_{12}\alpha_{21}.
\]
This relation defines a cone in the parameter space. On the other hand, by \eqref{eq:ODE2}, every linear $2\times 2$ system describing a first order reaction network with two species (hence, the reaction network has the complexes $X_1,X_2$ and possibly $0$), takes the form \[
\dot x =\begin{pmatrix}-\k_{21}-\k_{31} &\k_{12}\\ \k_{21}&-\k_{12}-\k_{32}\end{pmatrix}\,x + 
\begin{pmatrix} \k_{13} \\ \k_{23}\end{pmatrix},
\]
with non-negative $\k_{ij}$ ($\k_{ij}$ is zero if the corresponding reaction does not exist). The determinant condition on the Jacobian of the system simplifies to 
\[
\k_{21}\k_{32} + \k_{31}\k_{12} + \k_{31}\k_{32}=0 \quad \Leftrightarrow  \quad \k_{21}\k_{32} =\k_{31}\k_{12} =\k_{31}\k_{32}=0,
\]
due to non-negativity. In addition, the existence of stationary points requires conditions on $\k_{13}$ and $\k_{23}$. Evaluating the TFPV conditions, one sees that they all admit an interpretation in the reaction network framework: Certain reactions are being ``switched off''.  Furthermore, the conditions for TFPVs to exist yield very simple irreducible components of $W_s$, namely coordinate subspaces.
\end{example}

For a number of standard reaction networks in biochemistry (in particular those described in the first chapter of Keener and Sneyd \cite{KeSn}), all TFPVs were determined algorithmically in Goeke's dissertation \cite{godiss} and in the subsequent papers \cite{gwz,gwz3}.  It turned out that all of these admit an interpretation as a degenerate scenario in reaction network terms, via ``switched off'' reactions or missing species (and in some cases a combination of these). Based on these observations, and employing the theory of reaction networks, we will investigate conditions on reaction networks that guarantee the existence of singular perturbation scenarios.

For the reaction networks discussed in \cite{gwz,gwz3}, one finds that every irreducible component of $W_s$ (see \eqref{eq:Ws}) is just a coordinate subspace.
Indeed, it is not easy to find (realistic) systems where some component of $W_s$ is not  a coordinate subspace. 
This may be the case when non-stoichiometric first integrals exist for only some $\k$:
Let $s^*$ be the codimension of the reaction network (the number of independent stoichiometric first integrals). 
Assume the set $\widetilde{\Pi}$ of $\k$'s that give rise to extra linear first integrals is a proper algebraic variety and hence has measure zero (as for the reaction network in \eqref{eq:net1} in Example \ref{ex:stoich}). Any point in $\widetilde{\Pi}$ is a candidate for a  TFPV in dimension $s>s^*$, if furthermore the critical manifold {intersects the non-negative orthant} and is attracting. 
Going back to network \eqref{eq:net1}, the set $\widetilde{\Pi}$ consists of  TFPVs and is characterized by the condition $\k_1=\k_2$, as one easily verifies that there exists a linearly attracting critical manifold.
 
An artificial way to construct further examples where the set of TFPVs is not included in a coordinate subspace, is to consider any parametrized polynomial system for which the  dimension of the set of stationary points is larger than $s^*$ for some choice of parameters, and furthermore, all negative monomials of the $i$-th polynomial are multiples of $x_i$. The latter is enough to constructively  interpret the system as arising from a mass-action reaction network \cite{erdi-toth}, though the  networks obtained in this way are typically not realistic. The following example is generated in this way. 

\begin{example}\label{ex:minus}
Consider the following mass-action reaction network
\[ X_2 \cee{->[\k_1]} X_1+X_2 \cee{->[\k_2]} X_1,  \cee{->[\k_3]}  0 \qquad 2X_1  \cee{->[\k_4]} X_2+2X_1. 
\]
The associated ODE system in $\R^2_{\geq 0}$ is
\[ \dot{x}_1 = \k_1 x_2 - \k_3 x_1,\qquad \dot{x}_2 = -\k_2 x_1x_2 + \k_4 x_1^2 = x_1( -\k_2 x_2 + \k_4 x_1).\] 
Generically, the variety of stationary points consists of the point $(0,0)$ and has dimension $s^*=0$. 
However, when $\k_1\k_4=\k_2\k_3$, then the variety has dimension one and consists of the line $\k_1 x_2=\k_3 x_1$.  Additionally, a direct computation shows that the critical manifold is attracting for $(x_1,x_2)\in \R^2_{\geq 0}$. Hence, $\k_1\k_4=\k_2\k_3$ defines a set of TFPVs for dimension one. 
In this case, there are no linear first integrals. 
\end{example}

We now turn to system \eqref{eq:ODE1}, and first establish conditions to ensure that every TFPV lies in some proper coordinate subspace, thus every  irreducible component of $W_s$ is contained in some coordinate subspace.
If we require the critical manifold to intersect the positive orthant, the existence of TFPVs $\widehat\k\in \R^m_{>0}$ is easily precluded for important classes of reaction networks.

In preparation for Proposition~\ref{prop:noTFPV} we introduce some objects and some notation. Let $N'\in \R^{s^*\times m}$ consist of $s^*$ linearly independent rows of the stoichiometric matrix $N$. Moreover let $E\in \R^{m \times q}$ be a matrix whose columns are the extreme rays of the polyhedral cone $\ker(N)\cap \R^m_{\geq 0}$, and for $\lambda\in\R^m$ denote by $ \diag( E\lambda )$ the matrix with the entries of $E\lambda$ in the diagonal, and zeros off-diagonal. Consider the  matrix
\begin{equation}\label{eq:convex}
 N' \diag( E\lambda ) B^\top
\end{equation}
(with $B$ the reactant matrix, see Subsection~\ref{crnsubsec}). Finally, let $\Lambda$ be the set of $\lambda\in \R^q_{\geq 0}$ such that $E\lambda \in \R^m_{>0}$.  (The particular choice of $N'$ will be irrelevant.)

 \begin{proposition}\label{prop:noTFPV}
 Let $G$ be a mass-action reaction network of codimension $s^*$. With the notation introduced above, assume $G$ belongs to one of the following cases:
 \begin{enumerate}[(a)]
 \item  The set of positive stationary points $\mathcal{V}_\k\subseteq \R^n_{>0}$ admits a {smooth} parametrization of the form 
$\R^{s^*}_{>0}  \xrightarrow{\varphi_\k}  \mathcal{V}_\k$, 
with $\im (\varphi_\k)= \mathcal{V}_\k$ for all $\k\in \R^m_{>0}$.
\item The reaction network is \emph{injective} \cite{Feliu-inj}, hence the coefficient $\sigma_{n-s^*}(x,\k)$ of $\tau^{s^*}$ of the characteristic polynomial  of the Jacobian of system \eqref{eq:ODE1} is a polynomial in $x$ and $\k$ with only non-negative coefficients. 
\item For all $\lambda\in \Lambda$, at least  one of the minors of $N' \diag( E\lambda ) B^\top$ is non-zero.
\end{enumerate}
Then,  there are no TFPVs $\widehat\k\in \R^m_{>0}$ for which some {irreducible} component $Z$ of the critical variety intersects the positive orthant. 
  \end{proposition}

\begin{proof}
(a) The parametrization gives that the dimension of $\mathcal{V}_\k$ is $s^*$ for all $\k\in \R^m_{>0}$.  (b) For $\widehat\k$ to be a TFPV, we require $\sigma_{n-s^*}({x_0},\widehat\k)=0$ for some $x_0\in Z$, which occurs only  if  $\widehat\k$  or $x_0$  belong to a coordinate subspace. (c) The condition implies that the Jacobian of \eqref{eq:ODE1} has rank $n-s^*$ at any positive stationary point \cite{Pascual:ACR}. 
\end{proof}

\begin{remark} 
We make a few observations regarding the relevance of the criteria  in Proposition~\ref{prop:noTFPV}.
\begin{itemize}
\item Condition (c) holds for surprisingly many networks {and is computationally easy to verify}. When $\Lambda=\R^q_{>0}$, which occurs often, then condition (c) holds if there is one minor with all non-zero coefficients of the same sign. 
\item Many realistic reaction networks admit parametrizations in the sense of Proposition~\ref{prop:noTFPV}(a): Among these are reaction networks admitting toric steady states \cite{PerezMillan}, complex-balancing equilibria \cite{horja,feinberg-balance,Craciun-Sturmfels}  (see also Subsection~\ref{sec:CB}), and there are many reaction networks  for which parametrizations can be found using linear elimination of some variables in terms of the rest \cite{Fel_elim,saez_elimination} {(see \cite{FeliuPlos} for a short account on how to find parametrizations). }
\item Injective reaction networks admit at most one equilibrium in each SCC \cite{craciun-feinbergI,Feliu-inj}, and several criteria, in addition to the one stated in Proposition~\ref{prop:noTFPV}(b), have been established.  These criteria involve graphical conditions \cite{banaji-craciun1,banaji-craciun2} and sign vectors \cite{MullerSigns}.
\end{itemize}
\end{remark}

To include TFPVs with  critical manifold intersecting the positive orthant, it is appropriate (and necessary in the cases {covered in Proposition~\ref{prop:noTFPV}}) to deviate from the convention in Definition~\ref{crndef} and allow the rate parameters to be zero, thus change the parameter range of $\k$ to $\R_{\geq 0}^m$. Passing  from generic $\kappa\in\mathbb R^m_{>0}$ to a special $\widehat \kappa \in\mathbb R^m_{\geq 0}$  may be  seen as considering a subnetwork of the original reaction network.  To indicate this, we make the following definition.

\begin{definition}\label{rem:switching}
Let $\kappa\in\R_{\geq 0}^m$. We denote by $G(\k)$ the subnetwork obtained from $G$ by removing the reactions with indices in $\{1,\ldots,m\}\setminus \text{supp}(\k)$, that is, the $i$-th reaction is removed if $\k_i=0$, for $i=1,\dots,m$. Isolated nodes are not removed from  $G(\k)$, and hence  $G$ and $G(\k)$ have the same set of complexes and species. 
\end{definition}

Note that $G(\k)=G(\widetilde\k)$ as long as $ \text{supp}(\k)=\text{supp}(\widetilde\k)$.

\begin{proposition}\label{prop:generic_dimension}
Let $G$ be a mass-action reaction network of codimension $s^*$. Let $h(x,\k)$ denote the right-hand side of \eqref{eq:ODE2}. 
\begin{enumerate}[(a)]
\item Let $\k^*\in \R^m_{>0}$. If there exists $x^*\in \R^n_{>0}\cap \mathcal{V}(h(\cdot,\k^*))$  such that 
$D_1 h(x^*,\k^*)$ has rank   $n-s^*$ (respectively, additionally $n-s^*$ eigenvalues with negative real part), 
then  the same holds for a norm-open neighborhood of $\k^*$, and thus for a Zariski dense subset of $\R^m_{>0}$ containing $\k^*$. In particular, an irreducible component of $\mathcal{V}(h(\cdot,\k^*))$ has dimension $s^*$ and intersects the positive orthant. 
\item If $\k^*\in\mathbb R^m_{\geq 0}$  is a TFPV for dimension $s>s^*$ with $s$ the codimension of $G(\k^*)$, then the minimal   coordinate subspace containing $\k^*$ {is contained in} $W_s$. 
\end{enumerate}
\end{proposition}

\begin{proof} (a) 
We will use that $D_1 h(x^*,\k^*)$ has $n-s^*$ eigenvalues with negative real part, if and only if the corresponding $n-s^*$ Hurwitz determinants {of its characteristic polynomial, divided by $\tau^{s^*}$,} are positive, {see Gantmacher \cite[Ch. V, section 6]{Gant}.}
The rank of $D_1 h(x^*,\k^*)$ being $n-s^*$ implies that an irreducible component of $\mathcal{V}(h(\cdot,\k^*))$ has dimension $s^*$ and intersects the positive orthant \cite[\S9.6 Thm 9]{cox:little:shea}.
Let $V$ be the real algebraic variety in the variables $x,\k$ consisting of points where $h(x,\k)=0$ and $D_1 h(x,\k)$ has rank strictly smaller than $n-s^*$, respectively, at least one of the Hurwitz determinants vanishes. 
By hypothesis, there exists $(x^*,\k^*)\in \R^{n+m}_{>0} \setminus V$ satisfying  
$h(x^* ,\k^*)=0$. Let $U\subseteq \R^{n+m}_{>0} \setminus V$ be an open Euclidean ball containing $(x^*,\k^*)$. 
The intersection of $U$  and the zero set of $h$, which is non-empty, consists of points $(x,\k)$ such that $x\in \mathcal{V}(h(\cdot,\k))$ and the Jacobian has maximal rank $n-s^*$, respectively, all Hurwitz determinants are positive. 
The projection $\widehat{U}$ of $U$ onto $\R^m_{>0}$ in the variable  $\k$ contains a non-empty open Euclidian ball of parameters $\k_0$ for which there exists $x_0 \in \R^n_{>0}$ such that $D_1 h(x_0,\k_0)$ has rank   $n-s^*$, respectively, additional $n-s^*$ eigenvalues with negative real part.  {By the Implicit Function Theorem 
applied to  $h$ at $(x^*,\k^*)$, $\widehat{U}$  contains an open ball centred at $\k^*$ such that 
$ \R^n_{>0}\cap \mathcal{V}(h(\cdot,\k))\neq \emptyset$ for all $\k$ in the ball. As any Euclidean ball is Zariski dense, this concludes the proof of (a). }

(b) Let $C$ be the minimal   coordinate subspace containing $\k^*$. The parameter $\k^*$ and the reaction network $G(\k^*)$ satisfy the hypotheses of (a), after restricting $\R^m_{>0}$ to   $C$. Therefore, there exists {a norm-open and} Zariski dense set $U$ (relative to $C$) such that  any $\k'\in U\subseteq C\cap \R^m_{>0}$  is a TFPV for dimension $s$ {and thus $U\subseteq W_s$.}
Since $U$  is Zariski dense in $C$, its Zariski closure is $C$ {and it follows that $C\subseteq W_s$}. 
\end{proof}

Proposition~\ref{prop:generic_dimension}(b) does not imply that all rate parameters in the coordinate subspace are TFPVs given that one is a TFPV, but only that this is the case in an open set relative to the coordinate subspace. 
The next example illustrates this.

\begin{example}
Consider the following  mass-action reaction network,
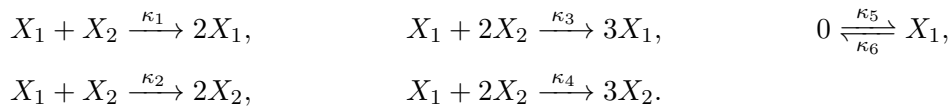
\begin{align*}
X_1+X_2 & \cee{->[\k_1]} 2X_1,  & X_1+2X_2 & \cee{->[\k_3]} 3X_1,  &  0  & \cee{<=>[\k_5][\k_6]} X_1, \\
X_1+X_2 & \cee{->[\k_2]} 2X_2, & X_1+2X_2 & \cee{->[\k_4]} 3X_2.
\end{align*}
The associated ODE system in $\R^2_{\geq 0}$ is
\begin{align*}
\dot{x}_1 &= (\k_1-\k_2)x_1x_2 + (2\k_3  -\k_4 )x_1 x_2^2 + \k_5 - \k_6 x_1,\\ 
\dot{x}_2 &= (-\k_1 + \k_2)x_1x_2  +(-2\k_3 + \k_4 )x_1x_2^2.
\end{align*}
The  codimension of the reaction network is $s^*=0$, and the reaction network has   one positive equilibrium
$(\tfrac{\k_5}{\k_6}, \tfrac{ \k_2-\k_1}{2\k_3-\k_4})$, provided the second entry is positive. 
Consider a parameter value of the form $\widehat\k=(\k_1,\k_2,\k_3,\k_4,0,0)$, {which corresponds to removing the pair of reactions $0  \cee{<=>} X_1$}. Then, the stoichiometric subspace of $G(\widehat\k)$  has codimension $s=1$, and the stationary variety consists of the two coordinate axes together with the line $x_2=\tfrac{\k_2-\k_1}{2\k_3-\k_4}$,  provided this expression is positive. In this case, there is a critical manifold of dimension one in $\R^2_{>0}$. 
The line intersecting the positive orthant is attracting if $\k_1>\k_2$ and repelling if $\k_2>\k_1$. Hence, $\widehat\k$ is a TFPV for dimension one if and only if $\k_1>\k_2$ and $\k_4>2\k_3$.  We also have that the 
minimal coordinate subspace $C$ containing $\widehat\k$ belongs to $W_1$.

If we now consider   $\widehat\k=(\k_1,0,0,\k_4,0,0)$, the codimension of $G(\widehat\k)$  is also $s=1$, and the positive part of the stationary variety consists of the  attracting line $x_2=\tfrac{\k_1}{\k_4}$. Hence,  
 $\widehat\k$ is a TFPV for dimension one. 
 The minimal coordinate subspace containing $\widehat\k$ is not an irreducible component of $W_1$, as  it is a proper Zariski closed set of  the coordinate subspace $C$.
 \end{example}

In what follows we consider   TFPVs   for two classes of reaction networks, namely first order reaction networks and complex-balanced reaction networks.  Due to special properties of the Laplacian matrix, TFPVs for complex-balanced reaction networks can be identified. 
Our results build on the understanding of the kernel of $A(\k)$ in \eqref{eq:ODE2}. Therefore, we first review  key results about Laplacian matrices and especially their kernel, using a graphical approach.

\subsection{Some properties of Laplacian matrices}

In this subsection we recall and review some properties of Laplacian and compartmental matrices.  For the following known facts refer e.g. to Jacquez and Simon \cite[Subsection 4.1]{JaSi}.

The Laplacian matrix of a directed graph (and thus the Laplacian $A(\k)$ of a reaction network) 
 is a compartmental matrix. We recall some notions.
\begin{itemize}
\item  A quadratic matrix with real entries is called a compartmental matrix if all its off-diagonal entries are $\geq 0$ and all its column sums are $\leq 0$. 
\item Given non-negative real numbers $\sigma_{ij}$, $1\leq i,\, j\leq n$, and $\tau_k$, $1\leq k\leq n$, the matrix
\begin{equation}\label{compmat}
L(\sigma, \tau):=
\begin{pmatrix}-\sum_\ell\sigma_{\ell 1 }-\tau_1& \sigma_{12} & \cdots & \sigma_{1d}\\
                                                                    \sigma_{21}&  & \ddots &\vdots\\
                                                                    \vdots& \ddots &  \ddots  & \sigma_{d-1,d}\\
                                                                     \sigma_{d1}&\cdots &  \sigma_{d,d-1} & -\sum_\ell\sigma_{ \ell d}-\tau_d\end{pmatrix}\in\mathbb R^{d\times d},
\end{equation}
with $\sigma=(\sigma_{ij})$ and $\tau=(\tau_k)$, is compartmental. In turn, every compartmental $d\times d$ matrix has a representation of the form \eqref{compmat}, with uniquely determined $\sigma_{ij}$ and $\tau_k$.
\item The Laplacian $A(\k)$ of a reaction network satisfies $\sigma=\k$ and $\tau=0$. Hence, $A(\k)=L(\k,0)$ and column sums are zero. \end{itemize}

\begin{lemma}\label{compmatlem} Let $L(\sigma,\tau)$ be a compartmental matrix as in \eqref{compmat}. Then,
all eigenvalues of $L(\sigma,\tau)$ have non-positive real part, and any eigenvalue with real part zero is equal to zero. Moreover $\mathbb R^n$ is the direct sum of the kernel and the image of $L(\sigma,\tau)$.
\end{lemma}

Consider a mass-action  reaction network $G$.  Let $G_1,\dots,G_r$  be the connected components of $G$ and further order the set of complexes 
 according to the  connected component they belong to. 
 If $A_i(\k)$ stands for the Laplacian matrix of $G_i$, then $A(\k)$ becomes a block diagonal matrix with $r$ blocks,
\[ A(\k)= \begin{pmatrix}
A_1(\k) & 0 & \dots & 0 \\
0 & A_2(\k)  & \dots & 0  \\
\vdots & \vdots & \ddots & \vdots \\
0 & 0 & \dots & A_r(\k) 
\end{pmatrix}\in \R^{d\times d}.\]

The form of the kernel of the Laplacian matrix $A(\k)$ of a digraph with $\k\in \R^m_{>0}$ is  well known, in particular, in the context of reaction networks \cite[Thm 16.4.2]{feinbergbook}. It derives from the Matrix-Tree theorem \cite{Tutte-matrixtree,mirzaev-laplacian,ChaikenKleitman}. 
\begin{itemize}
\item The dimension of the kernel of $A(\k)$ agrees with the number of terminal strongly connected components and is independent of $\k\in \R^m_{>0}$. 
\item If the digraph is strongly connected, then $\dim \ker A(\k)=1$ and a generator of $\ker A(\k)$ is given by the sequence of signed principal  minors (which are positive). 
\item If the digraph is not strongly connected, then any complex in the support of a vector in $\ker A(\k)$   belongs to a terminal strongly connected component. 
Furthermore, a basis of $\ker A(\k)$ can be chosen such that the support of each vector is exactly   one terminal strongly connected component and the non-zero entries  are positive. These entries arise as the signed principal  minors of the restriction of the matrix to the nodes in the component.
\item The vector $e=(1,\dots,1)$ belongs to the left-kernel of $A(\k)$, and generates it when the digraph has one terminal strongly connected class.
\end{itemize}

These facts lead to the following lemma. 

\begin{lemma}\label{lem:Laplacian}
Let $G$ be a mass-action reaction network with labelling $\k\in \R^m_{>0}$. 
Let $G_1,\dots,G_r$ be the connected components of $G$ and assume that the set of complexes is ordered in accordance with the components. Let $T$  be the number of terminal strongly connected components of $G$.

Then, the rank of $A(\k)$ does not depend on the choice of $\k\in \R^m_{>0}$, and in particular
\begin{enumerate}[(a)]
\item $\dim \ker A(\k)= T$.
\item $\ker A(\k)$ has non-trivial intersection with the  positive orthant $\mathbb R_{>0}^d$, if and only if $G$ is weakly reversible.
\item 
 The left-kernel of $A(\k)$ contains the following row vectors, one for each connected component: \[
\big(e^{(1)}, 0,\ldots,0\big), \ldots, \big(0,\ldots,0, e^{(r)}\big), \quad\text{with}\quad e^{(i)}=\left(1,\ldots,1\right)
\]
of size the number of nodes of $G_i$.
 If each connected component of $G$ has exactly one terminal strongly connected component, then these vectors span the left-kernel of $A(\k)$.
 \end{enumerate}
\end{lemma}
\begin{proof}
(a-b)  are direct consequences of the properties of the kernel of $A(\k)$ discussed above.
(c) The column sums in each block $A_i(\k)$ are zero, as each submatrix is a Laplacian. The second part follows from (a), as $T=r$.
\end{proof}

With the notation in Lemma~\ref{lem:Laplacian}, if a connected component of $G$ has more than one terminal strongly connected component, then the vectors given in Lemma~\ref{lem:Laplacian}(c) do not form a basis of the left-kernel of $A(\k)$. To obtain a basis, one has to augment them by  vectors that might depend on the particular entries of $A(\k)$, that is, on $\k$; see Example \ref{ex:stoich} for an illustration.

\subsection{Complex-balancing and TFPVs}\label{sec:CB}

We now turn to an important class of reation networks called complex-balanced reaction networks and the existence of TFPVs for this class.  Complex-balanced reaction networks are characterised by their equilibria,  called \emph{complex-balanced} equilibria. 
According to Horn and Jackson \cite{horja}  (see also Feinberg \cite[Ch.~15ff.]{feinbergbook}), a \emph{positive} equilibrium $z\in\mathbb R_{>0}^n$ of \eqref{eq:ODE2} is   complex-balanced for the parameter value $\kappa^*$ if 
\[A(\kappa^*) {z^Y}=0.\]
By Lemma~\ref{lem:Laplacian}(b), the existence of a positive complex-balanced equilibrium implies that the reaction network is weakly reversible \cite[Proposition 16.5.7]{feinbergbook}. However, weak reversibility is not a sufficient condition. 
Since $z$ is an equilibrium for  system \eqref{eq:ODE2} if and only if $YA(\kappa)  {z^Y}=0$, one needs to understand the relation between ${\rm Ker}\,YA(\kappa)$ and ${\rm Ker}\,A(\kappa)$. Obviously the latter is a subset of the former.

The following proposition gathers  well-known facts \cite[Thm 15.2.2,Thm 15.2.4, Lemma 16.3.1]{feinbergbook}.

\begin{proposition}\label{prop:cb}
Let $G=(\mathcal{Y},\mathcal{R},\k)$ be a mass-action reaction network with $d$ complexes, $r$ connected components, and codimension $s^*$.
\begin{itemize}
\item
Let $e_1,\ldots,e_d$ denote the standard basis of $\mathbb R^d$ and let $\Delta:=\left\{{e_j-e_i \mid\, Y_i\rightarrow Y_j}\in{\mathcal R}\right\}\subseteq \mathbb R^d.$
Then 
\[
{\rm Ker}\,YA(\kappa)={\rm Ker}\,A(\kappa)\oplus \left({\rm Ker}\,Y\cap {\rm span}\,\Delta\right).
\]
The dimension $\delta$ of ${\rm Ker}\,Y\cap {\rm span}\,\Delta$ is called the \emph{deficiency},
and satisfies $\delta=d- (n-s^*)- r$.
\item  If system \eqref{eq:ODE2} admits a complex-balanced equilibrium in $\mathbb R^n_{>0}$, then every   SCC  contains precisely one positive equilibrium, which also is complex-balanced, and the Jacobian has $n-s^*$ eigenvalues with negative real part and $s^*$ zero eigenvalues (counted with multiplicity).
  As a consequence, the positive equilibria of the system form a manifold of dimension $s^*$. 
\end{itemize}
\end{proposition}

If $G$ is weakly reversible and $\delta=0$, then all positive equilibria are complex-balanced, irrespective of the (positive) reaction rate constants. In general, there are $\delta$ algebraically independent relations on the rate parameters $\k$, characterizing when the reaction network admits positive complex-balanced equilibria.  These relations are explicit \cite{Craciun-Sturmfels,Dickenstein:2011p1112,feliu:node}. 
Complex-balanced equilibria form a manifold of dimension $s^*$, and the rank of the Jacobian of system \eqref{eq:ODE2} evaluated at the equilibrium is   $n-s^*$.

In the following theorem we use the notation $G(\widehat\k)$ introduced in Definition \ref{rem:switching}. 

\begin{theorem}\label{thm_redCB}
Let $G$ be a mass-action reaction network of  codimension $s^*$. 
Let $\widehat\k\in \R^m_{\geq 0}$ such that
\begin{itemize}
\item $G(\widehat\k)$ is weakly reversible of codimension   $s>s^*$.
\item The non-zero coordinates of $\widehat\k$ satisfy the relations for the existence of {positive} complex-balanced equilibria in $G(\widehat\k)$.
\end{itemize}
 Then, $\widehat\k$ is a TFPV for dimension $s$ of  system \eqref{eq:ODE2}. Furthermore, the minimal coordinate subspace containing $\widehat\k$ is contained in $W_s$. 
\end{theorem}
\begin{proof}
We verify properties (vi)-(iv) of TFPVs. By Proposition~\ref{prop:cb}, the dimension of the set of positive equilibria of $G(\widehat\k)$ is $s$. 
The remaining properties of a TFPV follow from the properties of complex-balanced equilibria in Proposition \ref{prop:cb}. The last statement follows from Proposition~\ref{prop:generic_dimension}(b).
\end{proof}

An immediate consequence of Theorem~\ref{thm_redCB} arises when $G$ is weakly reversible and has deficiency zero.
A key point is that the deficiency of any subnetwork obtained from $G$ by removing reactions
can only decrease \cite[Prop. 8.2]{joshi-shiu-III}. In particular, if $G$ has deficiency zero, then so does any subnetwork.

\begin{theorem}\label{cbthm}
Let a weakly reversible reaction network $G$ of deficiency zero be given, with dynamics governed by system \eqref{eq:ODE2}. 
Let $\widehat\kappa\in\mathbb R^m_{\geq 0}$ be such that the induced subnetwork $G(\widehat\k)$ is weakly reversible and has more connected components than $G$.   

Then $\widehat\kappa$ is a TFPV of system \eqref{eq:ODE2} for dimension $n-d+r$, with $d$ and $r$ the number of complexes, respectively, connected components of $G(\widehat\k)$. This dimension equals the codimension of $G(\widehat\k)$.
\end{theorem}
\begin{proof} 
Let $r^*$ be the number of connected components of $G$. 
As the deficiencies of $G$ and $G(\widehat\k)$ are zero, the codimensions of $G$ and $G(\widehat\k)$ are $s^*=n-d+r^*$ and $s=n-d+r$, respectively. As $r>r^*$, we have $s>s^*$. 
Furthermore, all parameter values $\widehat\k$ yield complex-balanced equilibria for $G(\widehat\k)$ as the deficiency is zero. The statement now follows from Theorem~\ref{thm_redCB}.  
\end{proof}

In particular, when the hypotheses of Theorem \ref{cbthm} hold, then $\widehat \kappa$ lies in a coordinate subspace of the parameter space. Moreover, the connected components of $G(\widehat\kappa)$ identify a 
coordinate subspace in $W_s$ for the appropriate $s>s^*$. 
For some classes of reaction networks, including weakly reversible reaction networks of deficiency zero, an explicit formula for the singular perturbation reduction was derived in \cite{fkw}.

\begin{example}
The reversible Michaelis-Menten system from Example~\ref{mmdisguiseex} has deficiency zero and codimension $s^*=2$.  
By Theorem~\ref{cbthm}, setting either $\k_1=\k_{-1}=0$ or $\k_2=\k_{-2}=0$, the number of connected components increases, and the resulting rate parameters are TFPVs for dimension $3$. 
\end{example}

\begin{example}\label{compinex1}

 The competitive inhibition reaction network with reversible product formation \cite{KeSn}, 
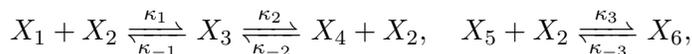
 \[X_1+X_2\ce{<=>[\k_1][\k_{-1}]}X_3\ce{<=>[\k_2][\k_{-2}]}X_4+X_2,\quad
X_5+X_2\ce{<=>[\k_3][\k_{-3}]}X_6,\]
has two additional reactions compared to the reversible Michaelis-Menten reaction network, see Example \ref{mmdisguiseex}, namely, inhibition of the enzyme ($X_2$) by an inhibitor ($X_5$) via formation of an intermediate complex ($X_6$).
The reaction network is weakly reversible with deficiency zero (five complexes, two linkage classes and stoichiometric subspace of dimension three). 
By Theorem~\ref{cbthm}, setting either $\k_1=\k_{-1}=0$, or $\k_2=\k_{-2}=0$, or $\k_3=\k_{-3}=0$, the number of connected components increases by one, and the resulting rate parameters are TFPVs for dimension $4$. In addition, choosing two of the three pairs to be zero, one obtains TFPVs for dimension $5$.

\end{example}
 
\begin{example}\label{futileex}
Consider the following reaction network, which is the futile cycle with one phosphorylation site \cite{Wang:2008dc}:
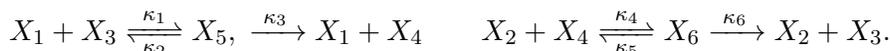
\[
X_1 + X_3 \cee{<=>[\k_1][\k_2]} X_5, \cee{->[\k_3]} X_1+X_4\qquad 
X_2 + X_4 \cee{<=>[\k_4][\k_5]} X_6 \cee{->[\k_6]} X_2+X_3.
\]
This reaction network is not weakly reversible and   has codimension $3$. An easy computation shows that the stationary set admits a parametrization with three free variables $x_1,x_2,x_3$, and hence has dimension $3$.  Proposition~\ref{prop:noTFPV}(a) applies.

Alternatively, Proposition~\ref{prop:noTFPV}(c) is applicable: For a specific choice of $N'$, the matrix $N' \diag(E\lambda) B^\top$ in \eqref{eq:convex} equals
\[\begin{pmatrix}
0 & -\lambda_{2} & 0 & -\lambda_{2} & \lambda_{2} & \lambda_{1}+\lambda_{2} 
\\
 \lambda_{1} & 0 & \lambda_{1} & 0 & -\lambda_{2}-\lambda_{3} & 0 
\\
 0 & \lambda_{2} & 0 & \lambda_{2} & 0 & -\lambda_{1}-2 \lambda_{2}-\lambda_{3} 
\end{pmatrix}.\]
We have $\Lambda= \R^3_{>0}$. The minor given by columns $1,2,5$ is $\lambda_1\lambda_2^2$, which is non-zero.

In conclusion, no TFPVs with positive entries and critical manifold intersecting the positive orthant exist.
Upon setting {$\k_3=\k_6=0$}, the resulting reaction network is weakly reversible and has deficiency $0$ with codimension $s=4$. Hence, by Theorem~\ref{thm_redCB}, any rate parameter of the form $(\k_1,\k_2,0,\k_4,\k_5,0)$ with non-zero entries being positive, is a TFPV
for dimension $4$. 
\end{example} 
 
 \begin{example}
 Consider  the following reaction network modelling an allosteric kinase \cite{feng:allosteric}:
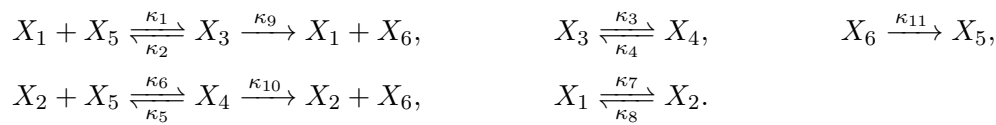
 \begin{align*}
 X_1 + X_5  & \cee{<=>[\k_1][\k_2]} X_3 \cee{->[\k_9]} X_1+X_6, & X_3 &  \cee{<=>[\k_3][\k_4]} X_4, & X_6 & \cee{->[\k_{11}]} X_5,   \\
  X_2 + X_5  & \cee{<=>[\k_6][\k_5]} X_4 \cee{->[\k_{10}]} X_2+X_6, & X_1 &  \cee{<=>[\k_7][\k_8]} X_2. 
 \end{align*}
The positive part of the stationary set admits a parametrization with $s^*=2$ free variables.  Upon setting $\k_9=\k_{10}=\k_{11}=0$, the reaction network becomes weakly reversible of deficiency $1$ and codimension $3$. 
The condition that characterizes when  complex-balanced equilibria exist, referred to in Theorem~\ref{thm_redCB}, is  $\k_2\k_4\k_6\k_7=\k_1\k_3\k_5\k_8$ \cite{Craciun-Sturmfels}. It follows that  any $\widehat\k=(\k_1,\k_2,\k_3,\k_4,\k_5,\k_6,\k_7,\k_8,0,0,0)$, fulfilling this condition,  is a TFPV for dimension $3$ and the corresponding coordinate subspace is included in $W_3$. 
 
 In this case, the positive part of the stationary variety of $G(\widehat\k)$ always admits a parametrization. Using Hurwitz determinants, one confirms that the variety is linearly attracting. Therefore, the whole positive part of this particular coordinate subspace is formed by TFPVs.
\end{example}

We conclude with an example of a weakly reversible reaction network admitting TFPVs that are not included in a proper coordinate subspace. 

\begin{example}
This example is introduced in \cite[Example 4.1]{boros:WR}, where  the purpose is to show the existence of weakly reversible reaction networks with infinitely many equilibria in some SCC. 
The reaction network consists of four connected components, written in rows for convenience:
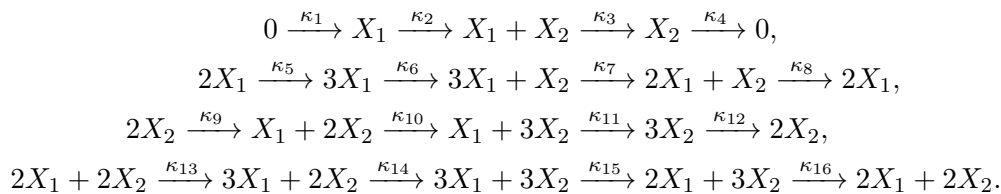
\begin{align*}
0  \ce{->[\k_1]} X_1 \ce{->[\k_2]} X_1+X_2 & \ce{->[\k_3]} X_2 \ce{->[\k_4]} 0, \\ 
2X_1  \ce{->[\k_5]} 3X_1 \ce{->[\k_6]} 3X_1+X_2 & \ce{->[\k_7]} 2X_1+X_2 \ce{->[\k_8]} 2X_1, \\
2X_2  \ce{->[\k_9]} X_1 + 2X_2 \ce{->[\k_{10}]} X_1+3X_2& \ce{->[\k_{11}]} 3X_2 \ce{->[\k_{12}]} 2X_2, \\ 
2X_1 + 2X_2 \ce{->[\k_{13}]} 3X_1 + 2X_2 \ce{->[\k_{14}]} 3X_1+3X_2 &\ce{->[\k_{15}]} 2X_1+3X_2 \ce{->[\k_{16}]} 2X_1 +2X_2.
\end{align*}
This reaction network is weakly reversible of  codimension  $s^*=0$. 
When all parameters are set to $1$ except for $\k_3=\k_8=\k_{10}=\k_{13}=a$, with $a>5$, then the stationary variety has dimension $1$: it consists of one unstable point $(1,1)$ and one attracting closed curve around $(1,1)$. Hence, any such rate parameter is a TFPV for dimension $1$, which does not belong to a proper coordinate subspace. One might note that the reaction network is of the form discussed in Example \ref{ex:minus} with all negative terms in the ODE system being multiples of $x_2$.
\end{example}


\subsection{TFPVs for first order reaction networks}

In this section, we consider the special case of a {mass-action} reaction network $G=(\mathcal{Y},\mathcal{R},\k)$  
containing only first order reactions; thus $d=n$ or $d=n+1$ and {non-zero} complexes may be identified with species. In the formulation
\begin{equation}\label{lineq}
\dot x=Y A(\k)\, x^Y,\qquad x\in \R^n_{\geq 0},
\end{equation}
the matrix $Y$ is simply the identity matrix if $0\notin\mathcal{Y}$ and the identity matrix with an extra zero column otherwise. Hence either
$x^Y=x$  or $x^Y=(x,1)^\top$. 

\begin{lemma}
A first order mass-action reaction network has deficiency zero.
\end{lemma}
\begin{proof}
With the notation introduced in Proposition \ref{prop:cb}, one has ${\rm Ker}\,Y\cap {\rm span}\,\Delta=\{0\}$, due to the form of $Y$. The assertion follows.
\end{proof}

\begin{remark}\label{rk:dimzeroset}
By Lemma~\ref{lem:Laplacian}, the rank of $A(\k)$ does not depend on $\k\in \R^m_{>0}$. Let $T$ be the number of terminal strongly connected components of $G$ and  $s^*$ be the codimension of $G$. 
We make the following observations:
\begin{itemize}
\item If $0\notin\mathcal{Y}$, then $Y$ is the identity matrix, and the solution set to 
\eqref{lineq} in $\R^n_{\geq 0}$ is $\ker A(\k)$ and $s^*=T$.
 \item If $0\in\mathcal{Y}$, then $s^*=T-1$. 
 \begin{itemize}
 \item If $0$ belongs to a terminal strongly connected component of $G$, then the solution set to 
\eqref{lineq}  in $\R^n_{\geq 0}$ is the linear  affine subspace of $\ker A(\k)\cap \R^n_{\geq 0}$ with last coordinate equal to $1$. 
 By the description of   $\ker A(\k)$ in Lemma~\ref{lem:Laplacian}, this subspace has dimension $T-1$.
   \item If  $0$ does not belong to a terminal strongly connected component of $G$, then 
\eqref{lineq} has no solution. Indeed, the last entry of $x^Y$ is equal to $1$, and hence positive, but  any vector in $\ker A(\k)$  has last entry zero.  
   \end{itemize}
\end{itemize}
\end{remark}

With this in mind, we obtain the following proposition.

\begin{proposition}\label{crnirred}
Let $A(\k)$ be the Laplacian matrix of a mass-action reaction network $G=(\mathcal{Y},\mathcal{R},\k)$  consisting only of first order reactions 
with dynamics governed by system 
\eqref{lineq} in $\R^n_{\geq 0}$. Let $T$ be the number of terminal strongly connected components of $G$, and $s^*$ the dimension of the solution set to  \eqref{lineq}.

\begin{enumerate}[(a)]
\item If $\widehat \k \in\mathbb R_{\geq 0}^m$ is a TFPV of \eqref{lineq} for dimension $s>s^*$, then $\widehat\k$ lies in a proper coordinate subspace of $\mathbb R^m$.
\item 
Let $\widehat \k \in\mathbb R_{\geq 0}^m$  be in a proper coordinate subspace of $\mathbb R^m$, and consider the subnetwork $G(\widehat\k)$.
Then $\widehat \k$ is a TFPV if and only if $G(\widehat\k)$ has more than $T$  terminal strongly connected components, and additionally the complex $0$ belongs to one such component, provided $0$ is a complex of  $G$. 
\item  Each irreducible component of $W_s$ for $s>s^*$  is a coordinate subspace of $\mathbb R^m$.
\end{enumerate}
\end{proposition}

\begin{proof} (a) The proof is straightforward as the dimension of the solution set to   
\eqref{lineq} 
does not depend on $\k$, provided all entries are positive. 
(b) We first make a digression. Consider $G(\widehat\k)$ and assume $0\in \mathcal{Y}$. Then the last column of the matrix $Y$ is zero and the last entry of $v(x)=x^Y$ is $1$. Let $\widetilde{A}(\widehat\k)$ be the submatrix of $A(\widehat\k)$ obtained by removing the last row and column. Let $\beta\in \R^{d-1}$ be the vector formed by the first $d-1$ entries of the last column of $A(\k)$.  Then $Y A(\widehat\k) v(x) =  \widetilde{A}(\widehat\k) x +\beta$. To prove (b), we apply Lemma~\ref{compmatlem} to the compartmental matrices $A(\widehat\k)$ or $\widetilde{A}(\widehat\k)$, depending on whether  $0$ is a complex of $G$.  (c) is a direct consequence of  (a) and (b), as $\Pi_s$ is a union of  coordinate subspaces of $\R^m_{\ge 0}$. 
\end{proof}

Rephrasing the statement of Proposition~\ref{crnirred}, all TFPVs 
are found by setting rate parameters to zero such that the number of terminal strongly connected components increases, and taking into consideration the role of the zero complex.
We note that the irreducible components of any $W_s$ can be identified by inspecting the graph $G$. 

If the considered first order reaction network $G$ in addition is weakly reversible, then  for this network Theorem \ref{thm_redCB} and Theorem \ref{cbthm} are both consequences of Proposition \ref{crnirred}. For  Theorem \ref{thm_redCB}, note that if the subnetwork $G(\widehat\k)$ of $G$ is weakly reversible with codimension $s>s^*$, then it must be that $\widehat\k\in\R^m_{\ge 0}$ belongs to a proper coordinate subspace of $\R^m$ and the number of terminal strongly connected components 
of $G(\widehat\k)$ exceeds the number of terminal strongly connected components of $G$. Hence, the conclusions of Theorem \ref{thm_redCB} follow from Proposition \ref{crnirred}(b),(c). Note that the second condition of Theorem \ref{thm_redCB} is trivially fulfilled because $G$ has deficiency zero, hence any subnetwork, in particular $G(\widehat\k)$,  has also deficiency zero \cite[Prop. 8.2]{joshi-shiu-III}.
For Theorem \ref{cbthm}, we remark that it is a consequence of Theorem~\ref{thm_redCB}, hence also of Proposition \ref{crnirred}. Alternatively, it follows directly from Proposition \ref{crnirred} by similar arguments to above.

\begin{example}
Consider a  first order reaction network with three complexes and four reactions,
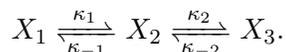
\[X_1 \cee{<=>[\k_1][\k_{-1}]} X_2  \cee{<=>[\k_2][\k_{-2}]}   X_3.\]
This reaction network has one terminal strongly connected component. 
By Remark~\ref{rk:dimzeroset}, $s^*=1$. 
There are three coordinate subspaces yielding  TFPVs for dimension $2$. These arise from the three ways to  increase the number of terminal strongly connected components: $\k_1=\k_{-1}=0$, or $\k_2=\k_{-2}=0$, or 
$\k_1=\k_{-2}=0$. 
\end{example}

\begin{example}
For the first order reaction network 
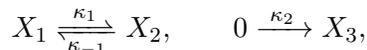
\[X_1 \cee{<=>[\k_1][\k_{-1}]} X_2,  \qquad 0 \cee{->[\k_2][]}   X_3,\]
we have two connected components and $s^*=1$  (Remark~\ref{rk:dimzeroset}), but this reaction network has no stationary points. Upon setting $\k_2=0$, we have three connected components and $0$ belongs to a terminal strongly connected component. Hence, by Proposition~\ref{crnirred}, 
$(\k_1,\k_{-1},0)$ is a TFPV for dimension $2$. 
\end{example}


 \section{Scalings, stoichiometry and TFPVs}\label{sec:Scale}

In this section, we start from LTC variable sets and the scaling approach to singular perturbation reductions of system  \eqref{eq:ODE2}, as initiated by Heineken et al.\ \cite{hta} (recall Subsection~\ref{slofasubsec} on slow-fast systems). A priori, there are no TFPVs that correspond to scalings, but these may appear when the system is restricted to {SCCs, as new parameters are introduced}.
For motivation, we look again at the reversible Michaelis-Menten system. 

\begin{example}\label{mmdisguiseexplus}
We continue Example \ref{mmdisguiseex}. The LTC variable set $\{x_2,\,x_3\}$ corresponds to the stoichiometric first integral $\phi_1=x_2+x_3$, and the LTC variable set $\{x_1,\,x_3,\,x_4\}$ corresponds to the stoichiometric first integral $\phi_2=x_1+x_3+x_4$. Moreover, on the SCC given by $x_2+x_3=e_0$ and $x_1+x_3+x_4=s_0$, one obtains the $2$-dimensional system
\begin{align*}
\dot x_1 &= -\k_1x_1(e_0-x_3)+ \k_{-1}x_3 & & \\
\dot x_3 &= \k_1x_1(e_0-x_3) - (\k_{-1}+\k_2)x_3+ \k_{-2}(e_0-x_3)(s_0-x_1-x_3).
\end{align*}
This system admits a TFPV with $e_0=0$, and all other parameters $>0$, with a degenerate (one dimensional) SCC forming the critical manifold, and a subsequent singular perturbation reduction. (For a TFPV with $s_0=0$, the SCC degenerates into a single point.)
\end{example}

Quite generally, LTC variable sets point to bifurcation scenarios, and possibly interesting dynamics may appear for small perturbations. In general there is no perfect correspondence to stoichiometric first integrals, as shown by examples in \cite{lawa}. But stoichiometric first integrals which correspond to LTC variable sets may, in turn, yield TFPVs of  the system on SCCs. 

 We start by characterizing LTC species sets. 

\subsection{A characterization of LTC species sets}

A useful modification of system \eqref{eq:ODE2} is the following, when some complexes are {\em non-reactant complexes},  that is, they only appear as product complexes. Complex $Y_j$ is non-reactant if and only if column $j$ of $A(\k)$ is zero. Thus, one may form $Y^*$ from $Y$, respectively,  $A^*(\k)$ from $A(\k)$, by removing all columns that correspond to indices of non-reactant complexes, to rewrite \eqref{eq:ODE2} as
\begin{equation}\label{eq:ODE2ac}
\dot x =Y A(\k) x^Y=Y A^*(\k)\,x^{Y^*}.
\end{equation}
Let $d^*$ be the number of reactant complexes, hence $Y^*\in \N_{0}^{n\times d^*}$. 
Note that $A^*(\k)$ is not a square matrix unless all complexes are reactant complexes and {thus} $A^*(\k)=A(\k)$.

We will first and foremost discuss  sets that are LTC species sets for {\em all} parameter values $\k\in\mathbb R_{>0}^m$.
The equations $x^Y=0$, respectively, $x^{Y^*}=0$ define varieties with coordinate subspaces as irreducible components, and the corresponding variables are obviously LTC variables. 
We will first show that all LTC variable sets of {system \eqref{eq:ODE2ac} }(which is the same as system \eqref{eq:ODE1}) are of this type.

The following proposition characterizes the LTC species sets. Recall  that reaction networks with inflow reaction do not admit any LTC species sets (remark below Definition \ref{def:indexset}).

\begin{proposition}\label{scaprop} 
Let system \eqref{eq:ODE1} be given.
Then $\{i_1,\ldots,i_u\}$ with $u<n$ and $1\leq i_1<i_2\cdots<i_u\leq n$ is an LTC index set for all $\kappa\in\mathbb R^m_{>0}$ if and only if 
\[
x^{Y^*}=0, \quad\text{whenever}\quad x_{i_1}=\cdots = x_{i_u}=0.
\]
\end{proposition}

\begin{proof}  The non-trivial assertion is the ``only if'' part. The ``if" part follows by definition. 
We need to show that if $h(x,\k)=N \diag(\k) x^B  =0$ for all $\kappa\in\mathbb R^m_{>0}$ whenever $ x_{i_1}=\cdots = x_{i_u}=0$, then also $x^{Y^*}=0$.
We may assume that the LTC index set is $\{1,\ldots,u\}$, and that complexes are ordered such that the first $d^*$ are reactant complexes. We let $y_{1},\ldots,y_{d^*}$ denote the columns of $Y^*$. 

We argue by contradiction and assume that some $x^{y_i}$, $i\in\{1,\ldots,d^*\}$, is non-zero when $x_1=\cdots=x_u=0$.  For $x_1=\cdots=x_u=0$ and $ i=1,\ldots,d^*$, we have
\begin{equation}\label{ueq}
x^{y_i}\not=0\quad\iff\quad y_i=\begin{pmatrix} 0 \\\ast\end{pmatrix},\quad\text{with}\quad 0 \in \mathbb R^u.
\end{equation}
We may assume that \eqref{ueq} holds precisely for the indices $d'\le i\le d^*$, for some  $d'\leq d^*$. 
Thus, we aim to show $d'=d^*$.

Let $K_\k$ be the $(m\times d^*)$-matrix with non-negative entries such that
\[ K_\k\, x^{Y^*} = \diag(\k) x^B,\qquad \textrm{hence}\qquad N \diag(\k) x^B = N K_\k\, x^{Y^*}.\]
Each entry of $K_\k$ is one of the rate parameters: the $(i,j)$-th entry is $\k_{\ell j}$ if the $i$-th reaction is $Y_j\rightarrow Y_\ell$. As $y_i$, $i=1,\ldots,d^*$, are pairwise different, the monomials $x^{y_i}$, $d'\le i \le d^*$, are linearly independent over $\mathbb R$. Using $N K_\k\, x^{Y^*}=0$, we obtain
\[
N K_\k \begin{pmatrix}0\\ \vdots\\0\\x^{y_{d'}}\\ \vdots \\ x^{y_{d^*}}\end{pmatrix}=0\qquad\Rightarrow \qquad N K_\k=\begin{pmatrix} \ast &\cdots&\ast&0&\cdots & 0\end{pmatrix},
\]
with the last $d^*-d'+1\geq 1$ columns equal to zero.   The equality tells us that the last $d^*-d'+1$ columns of $K_\k$ belong to $\ker(N)$ for all $\k\in \R^m_{>0}$.
The sum of these columns lies also in $\ker(N)$. The entries of the sum are positive when they correspond to reactions with $Y_{d'},\dots,Y_{d^*}$ among the reactant species, and zero otherwise. As $\k$ varies, we thus obtain a relatively open and non-empty subset in some proper coordinate subspace $C$ of $\R^m$. The row space of $N$ is therefore orthogonal to $C$; hence $N$ has at least one zero column, and we have reached a contradiction, as a reaction network does not have self-edges. 
\end{proof}

\begin{corollary}\label{ltcspecrem}
 LTC species sets are identifiable from the reactant complexes: A set of species $\{X_{i_1},\ldots,X_{i_u}\}$ is an LTC species set if and only if in every reactant complex, one of the $X_{i_k}$ appears with positive coefficient.
\end{corollary}

 An enumeration of all LTC species sets may start from those reactant complexes that contain the fewest species (that is, the  species appearing with positive stoichiometric coefficients). First, a species that appears alone in some reactant complex is necessarily contained in every LTC species set. Then proceed with complexes containing two species, and so on. From this  observation, one also finds that LTC species sets for first order reaction networks (with every complex consisting of one species) are comprised of all species in reactant complexes. Hence, the notions of LTC species and LTC variables are  of real interest only for non-linear systems.

\begin{example}\label{mmdisguiseex2} In Example
\ref{mmdisguiseex}, the reversible Michaelis-Menten reaction network, the reactant complexes are $X_1+X_2$, $X_3$ and $X_4+X_2$. Thus $X_3$ must lie in every LTC species set, and so must $X_2$ or $X_1$.
The first alternative yields the LTC species set $\{X_2,X_3\}$, while the second yields the LTC  species set $\{X_1,X_3, X_4\}$.  These are the only two LTC species sets.
In contrast, the standard irreversible Michaelis-Menten reaction network without the reaction $X_4+X_2\ce{->[\k_{-2}]}X_3$ has reactant complexes $X_1+X_2$ and  $X_3$, with two LTC species sets, $\{X_1,X_2\}$ and $\{X_2,X_3\}$.
\end{example}
 
\begin{example}\label{futileex2}
We consider again the futile cycle with one phosphorylation site, see Example \ref{futileex}:
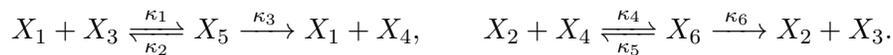
\[
X_1 + X_3 \cee{<=>[\k_1][\k_2]} X_5 \cee{->[\k_3]} X_1+X_4,\qquad 
X_2 + X_4 \cee{<=>[\k_4][\k_5]} X_6 \cee{->[\k_6]} X_2+X_3.
\]
Here, $X_5$ and $X_6$ are contained in every LTC species set, and altogether one finds the following LCT species sets,
\[
\{X_1,X_2,X_5,X_6\}, \quad\{X_1,X_4,X_5,X_6\},\quad\{X_2,X_3,X_5,X_6\},\quad \{X_3,X_4,X_5,X_6\}.
\]
Only the first and the last of these are also LTC species sets for the fully reversible system with the additional reactions $X_1+X_4\ce{->[\k_7]}X_5$ and $X_2+X_3\ce{->[\k_8]}X_6$.
\end{example} 

\subsection{LTC species and first integrals}
 
 We proceed to study the relation between LTC species sets and linear first integrals. 
 We first note a relation between LTC indices and the complex matrix. 

\begin{lemma}\label{ltcsmlem}
Let $\{i_1,\ldots,i_u\}$ with $u<n$ and $1\leq i_1<i_2\cdots<i_u\leq n$. 
Then the following statements are equivalent.
\begin{enumerate}[(a)]
\item $\{i_1,\ldots,i_u\}$ is an LTC index set.
\item The support of every column of $Y^*$ contains some $i_k$.
\item There exists a non-negative row vector $\omega \in \N^n_{0}$ with support $\{i_1,\ldots,i_u\}$ and such that every entry of $\omega \cdot Y^*$ is positive.
\end{enumerate}
\end{lemma}

\begin{proof} 
The equivalence of (a) and (b) is a restatement of Corollary \ref{ltcspecrem}. As for the equivalence of (b) and (c), note that 
\[
\omega \cdot Y^*=\left(\sum_{i=1}^{n} \omega_i\, y_{i,  1},\ldots,\sum_{i=1}^n \omega_i\, y_{i, d^*}\right)=\left(\sum_{k=1}^u \omega_{i_k}\, y_{{i_k},  1},\ldots,\sum_{k=1}^u \omega_{i_k}\, y_{{i_k} , d^*}\right). \]
Thus, the  $j$-th entry of $\omega\cdot Y^*$ is positive if and only if  $y_{{i_\ell},j}>0$ for some $i_\ell$.
As the $(i,j)$-entry of $Y^*$ is the stoichiometric  coefficient of $X_i$ in the complex $Y_j$, 
we have that $(\omega\cdot Y^*)_j>0$ if and only if the support of the $j$-th column of $Y^*$  intersects $\{i_1,\ldots,i_u\}$. 
The assertion follows.
\end{proof}

As a consequence of Lemma~\ref{ltcsmlem}, one finds that the support of certain stoichiometric first integrals consists of LTC indices.

\begin{proposition}\label{scacor}
Let $G$ be a mass-action reaction network {with $r$ connected components}, such that each connected component has one terminal strongly connected component. Assume that there exists a linear first integral $\phi=\sum_{i=1}^n \alpha_ix_i\not=0$, with non-negative integer coefficients.
\begin{enumerate}[(a)]
\item One has $\left(\alpha_1,\ldots,\alpha_n\right)\cdot Y\in \ker A(\k)$, and therefore, with the notation of Lemma~\ref{lem:Laplacian}, 
\[
\left(\alpha_1,\ldots,\alpha_n\right)\cdot Y=\sum_{i=1}^r \ell_i \left(0,\ldots,0,e^{(i)},0,\ldots,0\right), \qquad  \ell_i\in \N_{0}.
\]
\item 
If $\ell_i\neq   0$ for {all} $i=1,\ldots,r$, then the indices $i_1,\ldots,i_u$ in the support $\supp(\phi)$  form an LTC index set whenever $u<n$.
\end{enumerate}
\end{proposition}

\begin{proof}
 (a) Since the complexes  are pairwise different, the monomial entries of $x^Y$ are linearly independent over $\mathbb R$. Therefore
\[
\phi(YA(\kappa) x^Y)=0\quad\text{for all}\quad x\in\R^n_{\ge 0}\qquad\Leftrightarrow\qquad \phi(YA(\kappa))=0.
\]
Now, the statement follows from Lemma~\ref{lem:Laplacian}(c). 
(b) It follows directly from Lemma \ref{ltcsmlem}. 
\end{proof}

\begin{example} 
For the futile cycle from Examples \ref{futileex} and \ref{futileex2},
the linear first integral $\phi_1=x_1+x_2+x_5+x_6$ satisfies
\[ (1,1,0,0,1,1) \cdot Y =  (1,1,1,1,1,1)= (1,1,1,0,0,0)+(0,0,0,1,1,1),\]
where in the second equality  the vector is written as in Proposition~\ref{scacor} with $\ell_1=\ell_2=1$.
Hence, by  Proposition~\ref{scacor}, $\{X_1,X_2,X_5,X_6\}$ is an LTC species set. 
The linear first integral $\phi_2=x_1+x_5$ satisfies
\[ (1,0,0,0,1,0)\cdot  Y =  (1,1,1,0,0,0)= (1,1,1,0,0,0)+0\cdot (0,0,0,1,1,1),\]
and Proposition~\ref{scacor} does not apply. In fact, $\{X_1,X_5\}$ is not an LTC species set. 
The  LTC species set $\{X_1,X_4,X_5,X_6\}$ does not correspond to the support of any linear first integral, but it contains the support of one. 
\end{example}
\color{black}

The next example shows that non-negativity of the coefficients of the stoichiometric first integral in Proposition~\ref{scacor} cannot be discarded in general.

\begin{example} Consider the reversible Michaelis-Menten mass-action  reaction network {in Example~\ref{mmdisguiseex}}, with degradation of the intermediate complex (the reaction $X_3\ce{->[\k_3]}0$), governed by the ODE system,
\begin{align*}
\dot x_1 &= -\k_1x_1x_2 + \k_{-1}x_3  \\
\dot x_2 &= -\k_1x_1x_2 + (\k_{-1}+\k_2)x_3- \k_{-2}x_2x_4\\
\dot x_3 &=\k_1x_1x_2 - (\k_{-1}+\k_2)x_3+ \k_{-2}x_2x_4 -\k_3 x_3\\
\dot x_4 &= \k_2x_3 - \k_{-2}x_2x_4,
\end{align*}
with  $\k_3>0$.
As in the system without degradation (Example \ref{mmdisguiseex2}),  $\{X_2,\,X_3\}$ is an LTC species set, but the only stoichiometric first integral (up to multiples) is $\phi=x_1-x_2+x_4$, due to $(1,\,-1,\,0,\,1)\,Y=0$.  
The set $\{1,2,4\}$ is not an LTC index set. As noted in \cite{lawa}, the example also shows that the scaling approach may yield singular perturbation scenarios which are not directly related to TFPVs (even after restricting to SCCs).
One verifies that scaling $x_2$ and $x_3$ yields a system that admits a singular perturbation reduction to dimension two, with trivial reduced equation.  
\end{example}

\color{black}
\begin{remark}
There remains the question under which conditions the existence of LTC species sets in turn implies the existence of stoichiometric first integrals with corresponding support. We give a characterization for reaction networks with one connected component and one terminal strongly connected component.
Thus, let system \eqref{eq:ODE2} represent such a reaction network.
Assume without loss of generality that $\{X_1,\ldots,X_u\}$ is an LTC species set, and denote by {$\bar{y}_1,\ldots,\bar{y}_u$} the first rows of the complex matrix {$Y$}. By Lemma \ref{lem:Laplacian} the system admits a  stoichiometric first integral if, and only if, $e=(1,\ldots,1)$ is a multiple of some element in the closed convex hull of {$\bar{y}_1,\ldots,\bar{y}_u$}.  (Note that due to Lemma \ref{ltcsmlem}(c), there exist integers $\omega_1>0,\ldots,\omega_u>0$ such that {$\sum_{i=1}^u \omega_i \bar{y}_i> 0$} (coordinate-wise).)
\end{remark}

\subsection{Stoichiometry and TFPVs}\label{sec:LTC-TFPV}
We now address TFPVs of system \eqref{eq:ODE2} versus TFPVs of its restriction to stoichiometric compatibility classes. As seen in Example \ref{mmdisguiseexplus}, the restricted system may admit additional TFPVs. We first fix some notation.

We introduce the abbreviation
\begin{equation}\label{eq:ODE2h}
h(x,\k)=YA(\k)x^Y.
\end{equation}
In the following, we will assume that system \eqref{eq:ODE2h} admits a maximal set of independent stoichiometric first integrals $\phi_1,\ldots,\phi_{s^*}$. Then every SCC is the intersection of $\mathbb R_{\geq 0}^n$ with the common level set
\[
\phi_i(x)=\theta_i,\quad 1\leq i\leq s^*, 
\]
which we abbreviate as $S_\theta$, $ \theta=(\theta_1,\ldots,\theta_{s^*})$.
One may choose $\widehat x\in \mathbb R^{n-s^*}$ with entries from $x_1,\ldots,x_n$, such that {the Jacobian of } $(\widehat x,\, \phi_1(x),\ldots,\phi_{s^*}(x))$ has full rank $n$. 
This yields an equivalent version
\begin{equation}\label{eq:ODEscc}
\dot{\widehat x}=\widehat h(\widehat x,\k,\theta) \quad\text{in}\quad \mathbb R^{n-s^*},
\end{equation}
which for given $\theta$
represents system \eqref{eq:ODE2h} on $S_\theta$. 
We are interested in TFPVs  of the $(n-s^*)$-dimensional system \eqref{eq:ODEscc} for dimension $s>0$. Possible candidates for TFPVs are as follows.
\begin{itemize}
\item TFPVs via ``inheritance'' from \eqref{eq:ODE2h}: If $\widehat\kappa$ is a TFPV of \eqref{eq:ODE2h} for dimension $s>s^*$, then $(\widehat\kappa,\theta)$ is a TFPV of \eqref{eq:ODEscc} for some dimension $>0$ and some $\theta$ whenever a transversality condition is satisfied. This condition is rather weak: It does {\em not} hold only if the sum of the stoichiometric subspace  and the tangent space of $\mathcal V(h(\cdot,\widehat\kappa))$ at every $x_0\in \mathcal V(h(\cdot,\widehat\kappa))$ has dimension $<n-s^*+s$.
\item TFPV candidates from stoichiometric first integrals: Let the setting of Proposition \ref{scacor} be given and assume that the stoichiometric first integral $\phi_\ell$ with non-negative coefficients corresponds to an LTC variable set. If there exists a TFPV $(\widehat\kappa,\widehat\theta)$ with $\theta_\ell=0$, then the critical variety will be a coordinate subspace, and consequently by \cite{gwz3} the singular perturbation reduction will agree with the ``classical'' QSS reduction (in the sense of Subsection \ref{slofasubsec}) for the LTC variables. (We restrict attention to a single first integral here, since we are interested in minimal LTC sets; cf.\  Section \ref{slofasubsec}.)
\end{itemize}

There remains to establish manageable criteria for TFPVs from stoichiometric first integrals. The next result yields conditions for parameter values that are ``almost TFPV''.

\begin{proposition}\label{scctfp}
Let system \eqref{eq:ODE1} be given, and assume that every SCC of this system is compact {(equivalently, the left-kernel of $N$ in \eqref{eq:ODE1} has a vector with all entries positive \cite{benisrael})}.  Moreover assume that there exists a parameter $\widehat\theta\in\R^{s^*}$ such that: 
\begin{enumerate}[(a)]
\item No stationary points in $S_{\widehat\theta}$ are isolated relatively to  $S_{\widehat\theta}$. 
\item For every $\rho>0$, there exists some $\theta$ such that $\lVert \theta-\widehat\theta \rVert<\rho$ and $\dot{\widehat x}=\widehat h(\widehat x, \k,\theta)$ admits an isolated linearly attracting stationary point. (Here $\lVert\cdot\rVert$ denotes some norm.)
\end{enumerate}
Then, $\dot{\widehat x}=\widehat h(\widehat x, \k,\widehat\theta)$ admits a non-isolated stationary point whose Jacobian has only eigenvalues with non-positive real part, and admits zero as an eigenvalue.
\end{proposition}

\begin{proof} Given a compact subset $K$ of the parameter space, the union of the SCCs $S_\theta$ with $\theta\in K$ is compact. In the following, let $K$ be a compact neighborhood of $S_{\widehat\theta}$.

For every positive integer {$L$} let $\theta_L\in K$ be such that $\lVert\theta_L-\widehat\theta\rVert<1/L$ and $\dot{\widehat x}=\widehat h(\widehat x, \k,\theta_L)$ admits an isolated linearly attracting stationary point $\widehat z_L$. By compactness, the sequence $(\widehat z_L)_L$ in $\mathbb{R}^{n-s^*}$ has an accumulation point $\widehat z$, in  $S_{\widehat \theta}$. Since $\widehat z$ is not isolated, the Jacobian of $D_1\widehat h(\widehat z,\k,\widehat\theta)$ has the eigenvalue zero.
Moreover, the map which sends $(\widehat x,\k,\theta)$ to the coefficients of the characteristic polynomial
\begin{equation}\label{scccharpol}
\widehat\chi_{(\widehat x,\k,\theta)}{(\tau)}=\tau^{n-s^*}+\widehat \sigma_1(\widehat x,\k,\theta) \tau^{n-s^*-1}+\cdots+\widehat\sigma_{n-s^*}(\widehat x,\k,\theta)
\end{equation}
of $D_1\widehat h(\widehat x,\k,\theta)$ is continuous.
Thus, if some eigenvalue of the Jacobian had positive real part, the same would hold for some eigenvalue of the Jacobian of $D_1\widehat h(\widehat z_L,\k,\theta_L)$ with $L$ sufficiently large (see e.g. the reasoning in Gantmacher \cite{Gant}, Ch. V, section 3).
\end{proof} 

\begin{corollary}\label{sccdef0}
Assume that system \eqref{eq:ODEscc}  describes the dynamics of a weakly reversible deficiency zero reaction network. Then, the conclusion of Proposition \ref{scctfp} holds for every $\widehat\theta$, such that no stationary points are isolated in $S_{\widehat\theta}$.
\end{corollary}

\begin{remark} 
We clarify here what is meant by ``almost TFPV'' prior to the statement of  Proposition \ref{scctfp}. For this we discuss the conditions for TFPV in {Definition~\ref{def:TFPV} in} Subsection \ref{tfpvsubsec} for system \eqref{eq:ODEscc} and parameter value $\widehat\theta$.
\begin{itemize}
\item Condition (i) is always satisfied for some dimension $>0$, due to condition (a) in Proposition~\ref{scctfp}.
\item Condition (ii) requires equality of geometric and algebraic multiplicity for the eigenvalue $0$. This holds automatically when the algebraic multiplicity is equal to one. Generally this {property} can be checked by algebraic methods: For $\widehat x$ in the critical manifold, $\tau$ divides the characteristic polynomial  in \eqref{scccharpol}. {Obtain a new polynomial $\eta$ from $\widehat\chi_{(\widehat x,\k,\theta)}$  by} dividing out a power of $\tau$ such that a single factor $\tau$ remains. Then the multiplicity equals one if and only if $\eta$ annihilates $D_1\widehat h(\widehat x,\k,\widehat\theta)$.
But (as mentioned e.g.\ in \cite[Example 4]{gw2}) there exist realistic reaction networks for which the direct sum condition on the kernel and the image does not hold. 
\item Finally, to guarantee condition (iii), one needs to verify that there exist no purely imaginary eigenvalues except $0$. 
However,  if (iii) is not satisfied, then the system may admit some interesting dynamics, like zero-Hopf bifurcations. 
\end{itemize}
\end{remark}

We note a sharper result for the case of a one dimensional critical variety.

\begin{corollary}\label{corscctfp} In the setting of Proposition \ref{scctfp}, consider system \eqref{eq:ODEscc}, with characteristic polynomial of the Jacobian given by \eqref{scccharpol}. Let $\widehat \theta$ be such that $\widehat\sigma_{n-s^*}=0$ and $\widehat\sigma_{n-s^*-1}\not=0$ for $\theta=\widehat\theta$, some {$\kappa\in\R^m_{>0}$} and some stationary point $\widehat z\in{\R^{n-s^*}}_{\geq 0}$. Then it holds:
\begin{enumerate}[(a)]
\item The eigenvalue $0$ of $D_1\widehat h(\widehat z,\k,\widehat\theta)$ has multiplicity one.
\item There exists a polynomial $\Phi$ in $n-s^*-1$ variables with the following property: The Jacobian $D_1\widehat h(\widehat z,\k,\widehat\theta)$ admits non-zero purely imaginary eigenvalues if and only if $\Phi(\widehat\sigma_1,\ldots,\widehat\sigma_{n-s^*-1})=0$ at $(\widehat z,\k,\widehat\theta)$.
\end{enumerate}

\end{corollary}
\begin{proof}  (a) is obvious.  (b) There exists a polynomial $\Phi$ in the coefficients of the characteristic polynomial that vanishes if and only if a pair of (non-zero) eigenvalues adds up to zero; see e.g.\ \cite[Lemma 4.1, Appendix B]{kwhopf}. Since all eigenvalues have real part $\leq 0$, such a pair of eigenvalues must have zero real parts. 
\end{proof}

The non-trivial restrictions on the $\widehat\sigma_i$ in Corollary \ref{corscctfp}(b) suggest that there will be non-zero purely imaginary eigenvalues only in exceptional cases. 
There is an obvious (but less readily applicable) extension of Corollary~\ref{corscctfp} to TFPVs for dimension strictly larger than one, 
with an additional requirement that the geometric and the algebraic multiplicities of the zero eigenvalue are equal in Corollary  \ref{corscctfp}(a), and that Corollary \ref{corscctfp}(b) is left unchanged except for the number of { variables of $\Phi$}.

The polynomial $\Phi$ can be determined explicity. We recall some cases for SCCs of small generic dimension from \cite[Example 1]{kwhopf}.

\begin{remark}\label{remscctfp}
(a) If the SCCs of system \eqref{eq:ODE2} generically have dimension two, and the hypotheses of Proposition \ref{scctfp} are satisfied, then $\widehat\theta$ is a TFPV for dimension one whenever $\widehat \sigma_1\not=0$ at $(\widehat z, \k,\widehat\theta)$,

(b) If the SCCs of system \eqref{eq:ODE2} generically have dimension three, and the hypotheses of Proposition \ref{scctfp} are satisfied, then $\widehat\theta$ is a TFPV  for dimension one whenever  $\widehat \sigma_1\not=0$ and $\widehat \sigma_2\not=0$ at $(\widehat z,\k, \widehat\theta)$.

(c) If the SCCs of system \eqref{eq:ODE2} generically have dimension four, and the hypotheses of Proposition \ref{scctfp} are satisfied, then $\widehat\theta$ is a TFPV  for dimension one whenever  $\widehat \sigma_3\not=0$ and $\widehat\sigma_1\widehat \sigma_2\not=\widehat\sigma_3$ at $(\widehat z, \k,\widehat\theta)$.
\end{remark}

\begin{example}
Consider the reversible competitive inhibition reaction network in Example \ref{compinex1}, with 
the initial conditions $x_1(0)=s_0,\,x_2(0)=e_0,\,x_5(0)=i_0$ and $x_3(0)=x_4(0)=x_6(0)=0$,
and  $x_2+x_3+x_6=e_0$, $x_1+x_3+x_4=s_0$, $x_5+x_6=i_0$. Here, Corollary \ref{sccdef0} is applicable, with the critical parameter value $\widehat\theta$ having $e_0=0$, and  all other parameters being positive.

The dynamics on an SCC is described by
 the ODE system, 
\begin{align*}
{\dot x}_1 &=- \k_1(e_0-x_3-x_6)x_1+\k_{-1}x_3,   \\
{\dot x}_3&= \k_1(e_0-x_3-x_6)x_1-(\k_{-1}+\k_2)x_3 
                +\k_{-2}(e_0-x_3-x_6)(s_0-x_1-x_3), \\ 
{\dot x}_6&= \k_3 (e_0-x_3-x_6) (i_0-x_6)     - \k_{-3} x_6.
\end{align*}
 The Jacobian on the SCC with $e_0=0$ (thus, on the critical manifold with $x_3=x_6=0$) is equal to:
{\small \[
\begin{pmatrix}0 & \k_1x_1+\k_{-1} & \k_1x_1\\
                    0  & -(\k_1x_1+\k_{-1}+\k_2+\k_{-2}(s_0-x_1)) & -(\k_1x_1+\k_{-2}(s_0-x_1)) \\
                       0& -\k_3i_0 &  -(\k_3i_0+\k_{-3}) \end{pmatrix},
\]}
and the coefficients of its characteristic polynomial are 
\begin{align*}
\widehat\sigma_1&= i_0\k_3+\k_{-3}+\k_2+\k_{-1}+\k_1x_1+\k_{-2}(s_0-x_1),  \\
\widehat\sigma_2&=\k_3i_0(\k_2+\k_{-1}) + \k_{-3}(\k_1x_1+\k_{-1}+\k_2) + \k_{-2}\k_{-3} (s_0-x_1),\\
\widehat\sigma_3&=0.
\end{align*}
Since both $\widehat\sigma_1$ and $\widehat\sigma_2$ are positive when $0\leq x_1\leq s_0$,  the conditions in Remark \ref{remscctfp} and Corollary \ref{sccdef0} are satisfied, and $\widehat\theta$ is a TFPV for dimension one.
This system was discussed by elementary means in \cite{gswz}, with no reference to reaction network theory. A comparison shows that the approach developed here saves substantial computational effort. Moreover, one verifies that Proposition \ref{scctfp} is also applicable to the system with irreversible product formation (that is, $\k_{-2}=0$).

\end{example}

\subsection{Partial scalings: An outlook}
So far we considered on the one hand TFPVs that, in reaction network interpretation, arise from ``switching off'' certain reactions (Theorems \ref{thm_redCB} and \ref{cbthm}). On the other hand, we introduced LTC species sets, with characterizing property that if their concentrations are zero, then all reactions of the reaction network are precluded from taking place, and discussed their relation to TFPVs (Proposition \ref{scctfp} and Corollary \ref{corscctfp}). It is suggestive to combine these approaches by switching off certain reactions and determining LTC species for the remaining reactant complexes, and this combination will be sketched next. In the setting of Subsection \ref{slofasubsec} and in particular expansion \eqref{tayloreps}, we  consider LTC variable sets for a specific choice of the parameter {$\pi^*$ in \eqref{tayloreps}, for reaction networks $\pi^*=\widetilde \k$}. This yields a slow-fast system which may further admit a Tikhonov-Fenichel reduction. We will not attempt to establish necessary or sufficient conditions for a singular perturbation setting.

We start again from \eqref{eq:ODE2ac}, but now we consider some $\widetilde \k$ such that $A(\widetilde\k)$ has zero columns, thus there are additional non-reactant complexes  in the reaction network $G(\widetilde\k)$. We may assume that the remaining reactant complexes correspond to columns $y_1,\ldots,y_{\widetilde d}$ of $Y^*$, thus 
\[
A^*(\widetilde\k)=\begin{pmatrix} *&\cdots&*&0&\cdots&0\\
                                                  \vdots &&\vdots&\vdots&&\vdots\\
                                                   *&\cdots&*&0&\cdots&0\end{pmatrix}.
\]
The matrices of this type define some coordinate subspace of parameter space.  We denote by $Y_1$ the matrix with columns $y_1,\ldots,y_{\widetilde d}$, and by $Y_2$, the matrix with the remaining columns of $Y^*$. 
LTC variable sets for $G(\widetilde\k)$ can be identified via Proposition \ref{scaprop} with the complex matrix $Y_1$.

Upon relabelling, we may assume that $x_1,\ldots,x_u$ form an  LTC variable set for $G(\widetilde\k)$. 
Considering a curve $\varepsilon\mapsto \widetilde\k+\varepsilon\rho+\ldots$ in parameter space, we obtain
\[
A^*(\widetilde\k+\varepsilon\rho)=\begin{pmatrix} A_{11}+\varepsilon\cdots & \varepsilon A_{12}^*+\varepsilon^2\cdots\\
                                                                             A_{21}+\varepsilon\cdots & \varepsilon A_{22}^*+\varepsilon^2\cdots\end{pmatrix}
\]
with $A_{11}\in\mathbb R^{u\times \widetilde d}$. Moreover, set $x_i=\varepsilon x_i^*$ for $1\leq i\leq u$, then we have
\[
x^{y_j}=\varepsilon^{y_{1j}+\cdots+y_{uj}}{x_1^*}^{y_{1j}}\cdots {x_u^*}^{y_{uj}}\cdot x_{u+1}^{y_{u+1,j}}\cdots x_n^{y_{nj}},
\]
noting that the exponent of $\varepsilon$ is positive for all $j\leq \widetilde d$.
Abbreviating $w_1(x)=x^{Y_1}$, and $w_2(x)=x^{Y_2}$, one obtains an expansion
\[
\begin{array}{rcl}
w_1(\varepsilon x_1^*,\ldots,\varepsilon x_u^*,x_{u+1},\ldots,x_n)&=&\varepsilon w_1^*(x_1^*,\ldots,x_u^*,x_{u+1},\ldots,x_n)+\varepsilon^2\cdots,\\
w_2(\varepsilon x_1^*,\ldots,\varepsilon x_u^*,x_{u+1},\ldots,x_n)&=&w_2^*(x_1^*,\ldots,x_u^*,x_{u+1},\ldots,x_n) + \varepsilon\cdots,
\end{array}
\]
and altogether we arrive at the slow-fast system,
\begin{equation*}
\frac{d}{dt}\begin{pmatrix} x_1^*\\ \vdots\\ x_u^*\\x_{u+1}\\ \vdots\\ x_n\end{pmatrix}=\begin{pmatrix} A_{11}&A_{12}^*\\ \varepsilon A_{21}&\varepsilon A_{22}^*\end{pmatrix}\cdot \begin{pmatrix} w_1^*(x_1^*,\ldots,x_u^*,x_{u+1},\ldots,x_n) \\ w_2^*(x_1^*,\ldots,x_u^*,x_{u+1},\ldots,x_n)\end{pmatrix}+\begin{pmatrix} \varepsilon \cdots\\ \varepsilon^2\cdots\end{pmatrix}.
\end{equation*}

One would arrive at the same slow-fast system by starting from a different vantage point: First designate LTC variables and then switch off all reactions whose source complexes do not contain these variables.

Since the fast part of the scaled system involves slow reactions corresponding to $A_{12}^*$, the results from the previous subsections do not carry over to partial scalings. We will not discuss these matters any further here.

To close the present paper, we are satisfied to indicate by example that this heuristic is worth pursuing for general reaction networks.

\begin{example}
We continue the Michaelis-Menten reaction network from Example \ref{mmdisguiseex}, rewriting \eqref{mmcrn} in the form \eqref{eq:ODE2} with
\[
Y=\begin{pmatrix} 1&0&0\\ 1&0&1\\0&1&0\\ 0&0&1\end{pmatrix},\quad A(\k)=\begin{pmatrix} -\k_1 & \k_{-1}& 0\\   \k_1 &-( \k_{-1}+\k_2)& \k_{-2}\\  0 & \k_2& -\k_{-2}\end{pmatrix}, \quad x^Y=\begin{pmatrix}x_1x_2\\ x_3\\ x_2x_4\end{pmatrix}.
\]

{If $\widetilde{\k}$ is such that $\{x_3\}$ is an LTC variable set for $G(\widetilde{\k})$, then we need $\k_1=\k_{-2}=0$ at $\varepsilon=0$ (slow formation of the intermediate complex from both sides). Then for the curve in parameter space $\widetilde{\k} + \epsilon \k^* = ( \epsilon \k_1^*, \k_{-1} , \k_2 , \epsilon \k_{-1}^*)$,   }
\[
A(\k)=\begin{pmatrix} 0 & \k_{-1}& 0\\   0 &-( \k_{-1}+\k_2)& 0\\  0 & \k_2& 0\end{pmatrix}+\varepsilon \begin{pmatrix} -\k_1^* & 0& 0\\   \k_1^* &0 & \k_{-2}^*\\  0 & 0& -\k_{-2}^*\end{pmatrix},
\]
and scaling $x_3=\varepsilon x_3^*$ yields
\begin{align*}
\dot x_1&=\varepsilon (-\k_1^*x_1x_2+\k_{-1}x_3^*), \\
\dot x_2&= \varepsilon(-\k_1^*x_1x_2+\k_{-1}cx_3^*+\k_2x_3 ^*- \k_{-2^*}x_2x_4),\\
\dot x_3^*&=\k_1^*x_1x_2-\k_{-1}x_3^* - \k_2x_3^*+ \k_{-2}^*x_2x_4,\\
\dot x_4&=\varepsilon(\k_2x_3^* - \k_{-2}^*x_2x_4).\\
\end{align*}

Here, Tikhonov's theorem is directly applicable, with the reduced system admitting the first integrals $\phi_1=x_2$ and $\phi_2=x_1+x_4$. Thus, we end up with a one-dimensional equation (see \cite{gwz}).

Designating the LTC variable set $\{x_2\}$ forces  $\k_{-1}=\k_2=0$ at $\varepsilon=0$ (slow degradation of the intermediate complex in both directions). Proceeding as before, one has
\[
A(\k)=\begin{pmatrix} -\k_1 & 0& 0\\   \k_1 &0& \k_{-2}\\  0 & 0& -\k_{-2}\end{pmatrix}+\varepsilon \begin{pmatrix} 0 & \k_{-1}^*& 0\\   0 &-( \k_{-1}^*+\k_2^*)& 0\\  0 & \k_2^*& 0\end{pmatrix}.
\]
Scaling $x_2=\varepsilon x_2^*$, one obtains
\begin{align*}
\dot x_1&=\varepsilon (-\k_1x_2^*x_1+\k_{-1}^*x_3), \\
\dot x_2^*&= -\k_1x_2^*x_1+\k_{-1}^*cx_3+\k_2^*x_3- \k_{-2}x_2^*x_4,\\
\dot x_3&=\varepsilon( \k_1x_2^*x_1-\k_{-1}^*x_3 -\k_2^*x_3+ \k_{-2}x_2^*x_4),\\
\dot x_4&=\varepsilon(\k_2^*x_3- \k_{-2}x_2^*x_4),
\end{align*}
for which, again, Tikhonov's theorem is directly applicable. The reduced system admits the first integrals $\phi_1=x_3$ and $\phi_2=x_1+x_4$. Thus, again one arrives at a reduction to dimension one; see \cite{gwz} for details. 

{Hence, as noted earlier, for} reversible Michaelis-Menten, by the approaches in the present paper we have obtained all TFPVs that were determined algorithmically in  \cite{gwz} for this system.
\end{example}

\bigskip

\noindent{\bf Acknowledgements.} SW acknowledges support  by the bilateral project ANR-17-CE40-0036 and DFG-391322026 SYMBIONT. EF  acknowledges support from the Independent Research Fund  Denmark, and by the Novo Nordisk Foundation (Denmark), grant NNF18OC0052483. CW acknowledges support from the Novo Nordisk Foundation (Denmark), grant NNF19OC0058354. 
 
{\small

}
\end{document}